\documentclass[hidelinks,onefignum,onetabnum,final]{siamart220329}

\usepackage{lipsum}
\usepackage{amsfonts}
\usepackage{graphicx}
\usepackage{epstopdf}
\usepackage{algorithmic}
\usepackage{enumitem}
\usepackage{amsmath} 
\usepackage{amssymb}
\usepackage{mathtools}
\usepackage{amssymb}
\ifpdf
  \DeclareGraphicsExtensions{.eps,.pdf,.png,.jpg}
\else
  \DeclareGraphicsExtensions{.eps}
\fi

\newsiamremark{remark}{Remark}
\newsiamremark{hypothesis}{Hypothesis}
\newsiamremark{example}{Example}
\crefname{hypothesis}{Hypothesis}{Hypotheses}
\newsiamthm{claim}{Claim}

\title{A quasi-Grassmannian gradient flow model for eigenvalue problems\thanks{Submitted to the editors DATE.
\funding{This work was supported by the National Natural Science Foundation of China under
	grant 12571446 and the National Key R \& D Program of China under grants 2025YFA1016600 and 2025YFA1016601.}}}

\author{Shengyue Wang\thanks{SKLMS, State Key Laboratory of Mathematical Sciences, Academy of Mathematics and Systems Science, Chinese Academy of Sciences, Beijing 100190, China; and School of Mathematical Sciences, University of Chinese Academy of Sciences, Beijing 100049, China (\email{wangshengyue@amss.ac.cn}, \email{azhou@lsec.cc.ac.cn}).} 
\and Aihui Zhou\footnotemark[2]}

\usepackage{amsopn}

\begin{document}
	
\maketitle

	\begin{abstract}
	We propose a quasi-Grassmannian gradient flow model for eigenvalue problems of linear operators, aiming to efficiently address many eigenpairs. Our model inherently ensures asymptotic orthogonality: without the need for initial orthogonality, the solution naturally evolves toward being orthogonal over time. We establish the well-posedness of the model, and provide the analytic representation of solutions. Through asymptotic analysis, we show that the gradient converges exponentially to zero and that the energy converges exponentially to its minimum. This implies that the solution of the quasi-Grassmannian gradient flow model converges to the solution of the eigenvalue problems as time progresses. These results provide a continuous-flow framework in which the Stiefel constraint is recovered asymptotically rather than imposed on the initial data.
\end{abstract}


\section{Introduction}
Eigenvalue problems are fundamental in science and engineering, with applications in fields such as photonic crystals \cite{dorfler2011photonic}, uncertainty quantification \cite{smith2024uncertainty}, and quantum theory \cite{dirac1929quantum}. For instance, the computation of an approximate Karhunen-Lo{\`e}ve (KL) expansion requires solving very large eigenvalue problems to obtain many KL eigenpairs \cite{breuer1991use,schwab2006karhunen}. Similarly, electronic structures are often modeled by Hartree-Fock equations or Kohn-Sham equations, which also involve many eigenpairs and can be reduced to repeatedly solving large-scale linearized eigenvalue problems through discretization and self-consistent field iterations \cite{cances2023density,chen2014adaptive,chen2013numerical,lebris2005computational,martin2020electronic,saad2010numerical}.

The computational cost of solving large-scale eigenvalue problems is significant \cite{golub2013matrix,saad1992numerical,sholl2009density}, particularly when a substantial number of eigenpairs are required. A major bottleneck arises from orthogonalization procedures in solving algebraic eigenvalue problems, whose computational complexity grows dramatically with the problem size. 
We understand that an alternative approach, particularly for approximations in Kohn-Sham density functional theory (DFT), is solving the energy minimization problem \cite{dai2017conjugate,Dai2021,payne1992iterative,schneider2009direct,Zhang2014,Zhao2015}. Notably, even in these approaches, orthogonalization operations are generally unavoidable in computation due to the imposition of orthogonality constraints.

{To address this computational challenge, Dai et al.\ proposed an extended gradient-flow model for Kohn--Sham DFT that preserves orthogonality throughout the evolution and thereby avoids explicit orthogonalization in computation \cite{dai2020,dai2021convergent}. Building on this framework, Chu et al.~\cite{chu2025orthogonality} developed an accelerated method for eigenvalue problems while retaining this structural property.}
These orthogonality-preserving approaches are built on orthogonal initial data and on an evolution constrained by the Stiefel structure. 
This observation motivates the present work: we seek a continuous-flow formulation that allows non-orthogonal initial data while still recovering the orthogonality structure asymptotically.

In this paper, we propose a continuous quasi-Grassmannian gradient flow model for computing many eigenpairs of a self-adjoint operator in an infinite-dimensional setting. By extending the Grassmann gradient from the constraint manifold to the ambient space, the model allows non-orthogonal initial data and drives the solution asymptotically toward the Stiefel constraint. Thus, exact initial orthogonality is not required in the continuous-flow framework.
	
	The main contributions of this work are summarized as follows.
	\begin{itemize}
		\item We formulate a quasi-Grassmannian gradient flow model and show that the model asymptotically recovers the Stiefel constraint. In addition, Theorem~4.3 proves the asymptotic orthogonality of the solution.
		
		\item We establish the well-posedness of the model in the infinite-dimensional operator setting. Theorem~3.6 proves the existence and uniqueness of weak solutions. In particular, we derive an analytic representation of the solution in Theorem~3.10.
		
		\item We analyze the long-time behavior of the flow. Theorem~4.6 proves the exponential decay of the Grassmann gradient. Under the additional admissibility assumptions in Theorem~4.9, the associated energy converges to its minimum and the equivalence class $[U(t)]$ converges to the eigenspace associated with the first $N$ eigenvalues.
	\end{itemize}

While this paper focuses on a quasi-Grassmannian gradient flow model for operator eigenvalue problems, related gradient-flow ideas also arise in spectral analysis and other application areas. In an infinite-dimensional spectral setting, Mazzoleni and Savare \cite{MazzoleniSavare2023} investigated $L^2$-gradient flows of spectral functionals depending on the eigenvalues of Schr\"odinger-type operators. Gradient-flow methods have been developed for computing ground state solutions of Bose--Einstein condensates (BEC) \cite{Bao2004,chen2024convergence,chen2024fully,henning2020sobolev}. Specially, gradient-flow viewpoints are central in machine learning and neural networks \cite{ambrosio2008gradient,bottou2018optimization}, where they provide a fundamental framework for understanding optimization dynamics. A classical matrix example is Oja's work \cite{oja1982simplified,oja1989neural} on principal component extraction by unsupervised neural networks; the associated learning rules are often formulated as nonlinear differential or difference equations of matrices \cite{oja1985stochastic}.

	The corresponding Oja flow arising from the matrix setting has been widely studied, with convergence and stability results established in \cite{chen1998global,yan1994global}. In fact, the analysis of Oja's flow can be incorporated into our general framework. The present model accommodates such matrix-based flows under weaker initial orthogonality requirements and extends the flow formulation from matrices to infinite-dimensional operators. We also mention the spectral sensitivity analysis of Hiriart-Urruty and Ye \cite{HiriartUrrutyYe1995}, which concerns all eigenvalues of a symmetric matrix and is related to the present work through spectral variational analysis in the matrix setting rather than through gradient-flow dynamics. In contrast to spectral flows driven by potentials or other coefficients, our model evolves eigenfunction approximations directly in the ambient space and proves that the Stiefel constraint is recovered asymptotically from non-orthogonal initial data.

The remainder of the paper is organized as follows. Section 2 introduces the fundamental notation and formulates the operator eigenvalue problem. Section 3 presents the quasi-Grassmannian gradient flow model, analyzing its well-posedness, and deriving the solutions' analytic representation. Section 4 focuses on the asymptotic analysis, including asymptotic orthogonality, convergence to equilibrium, and convergence to the minimizer. Section 5 concludes the paper and outlines directions for our ongoing and future work. Finally, the Appendix provides the detailed proofs for the proposition, corollaries, and lemmas used in the paper.
\section{Preliminaries}
In this section, we introduce the basic notation and preliminary lemmas, and discuss the eigenvalue problem of infinite-dimensional linear operators.
\subsection{Notation and preliminary lemma}
\subsubsection{Notation} 
We denote by $W^{s, p}(\Omega)$ $(p\geqslant 1, s\geqslant 0)$ the standard Sobolev spaces with the induced norm $\|\cdot\|_{s, p}$, where $\Omega \subset \mathbb{R}^d$ $(d\in \mathbb{N}_+)$ is a bounded domain with a regular boundary (see, e.g. \cite{evans2022partial,zeidler2013nonlinear}). For $p=2$, we denote by $H^s(\Omega)=W^{s, 2}(\Omega)$ with the norm $\|\cdot\|_{s}=\|\cdot\|_{s, 2}$ and 
$$H_0^1(\Omega)=\left\{u\in H^1(\Omega):\left.u\right|_{\partial \Omega}=0\right\},$$ where $\left.u\right|_{\partial \Omega}=0$ is understood in the sense of trace. 

Consider the Hilbert spaces
\begin{equation*}
	\begin{aligned}
			\big(L^2(\Omega)\big)^N = \left\{\left(u_1, u_2, \cdots, u_N\right): u_i\in L^2(\Omega), i = 1, 2, \cdots, N\right\},
	\end{aligned}
\end{equation*}
and
\begin{equation*}
	\begin{aligned}
			\big(H_0^1(\Omega)\big)^N = \left\{\left(u_1, u_2, \cdots, u_N\right): u_i\in H_0^1(\Omega), i = 1, 2, \cdots, N\right\}.
	\end{aligned}
\end{equation*}
Let 
\begin{equation*}
	\begin{aligned}
		U=\left(u_1, u_2, \cdots, u_N\right) \text{ and }V=\left(v_1, v_2, \cdots, v_N\right).
	\end{aligned}
\end{equation*}
We introduce the $L^2$ inner product matrix as follows
\begin{equation*}
	U^{\top} V=\Big(\left(u_i, v_j\right)\Big)_{i, j=1}^N \in \mathbb{R}^{N \times N},\quad \forall U, V  \in 	\big(L^2(\Omega)\big)^N, 
\end{equation*}
where $(u_i, v_j)= \int_{\Omega} u_i(x) v_j(x) \mathrm{d}x$. The inner product $(\cdot, \cdot)$ and its associated norm are then defined as
\begin{equation*}
	\begin{aligned}
		(U, V)=\operatorname{tr}\left( U^{\top} V\right), \quad \| U \|=\sqrt{(U, U)},\quad \forall U, V  \in 	\big(L^2(\Omega)\big)^N.
	\end{aligned}
\end{equation*}
Similarly, the inner product $(\cdot, \cdot)_1$ on $\big(H_0^1(\Omega)\big)^N$ and the corresponding norm are given by
\begin{equation*}
	\begin{aligned}
		(U, V)_1=\operatorname{tr}\left( U^{\top} V +\left(\nabla U \right)^{\top} \nabla V  \right), \quad\| U \|_1=\sqrt{(U, U)_1}, \quad \forall U, V  \in\big(H_0^1(\Omega)\big)^N.
	\end{aligned}
\end{equation*}
Here $\nabla U$ is obtained by taking the gradient of each component of $U$, that is, $$\nabla U = \left(\nabla u_1, \nabla u_2, \cdots, \nabla u_N\right).$$

For $\mathcal{F}=\left(\mathcal{F}_1, \mathcal{F}_2, \cdots, \mathcal{F}_N\right) \in\left(\big(H_0^1(\Omega)\big)^N\right)^{\prime}=\big(H^{-1}(\Omega)\big)^N$, we set
$$
\mathcal{F}^\top U=\Big( \left\langle\mathcal{F}_i, u_j\right\rangle\Big)_{i, j=1}^N \in \mathbb{R}^{N \times N}, \quad  \forall U \in \left(H_0^1\left(\Omega\right)\right)^N,
$$
where $\langle\cdot, \cdot\rangle$ is the duality pairing of $H^{-1}(\Omega)$ and $H_0^1(\Omega)$. We then define
$$
\langle\mathcal{F}, U\rangle  = \operatorname{tr}\left(\mathcal{F}^\top U \right),\quad  \forall U \in \left(H_0^1\left(\Omega\right)\right)^N,
$$
and the norm
\begin{equation*}
	\begin{aligned}
		\|\mathcal{F}\|_{-1}= \sup\limits_{U \in \big(H_0^1(\Omega)\big)^N \backslash \{0\}} \frac{\left|\langle\mathcal{F}, U\rangle \right|}{\| U \|_1}.
	\end{aligned}
\end{equation*}

We introduce the Stiefel manifold as follows
$$	\mathcal{M}^N=\left\{U \in \big(H_0^1(\Omega)\big)^N: U^\top U=I_N\right\},$$
and it is clear to see that
$$U \in \mathcal{M}^N \iff UQ \in \mathcal{M}^N,\quad \forall Q \in \mathcal{O}^N,$$	
where $\mathcal{O}^N$ denotes the set of orthogonal matrices of order $N$. 

Let $\mathcal{G}^N$ be the Grassmann manifold (also known as the Grassmannian), which is a quotient manifold of the Stiefel manifold, that is
$$\mathcal{G}^N = \mathcal{M}^N/ \sim, $$
where $\sim$ denotes the equivalence relation defined by $\hat{U} \sim U$ if and only if there exists $Q \in \mathcal{O}^N$ such that $\hat{U}=U Q$. For any $U \in \big(L^2(\Omega)\big)^N$, we denote the equivalence class as
$$
[U]=\left\{U Q: Q \in \mathcal{O}^N\right\},
$$
then $\mathcal{G}^N$ can be formulated as
$$
\mathcal{G}^N=\left\{[U]: U \in \mathcal{M}^N\right\}.
$$

We define the $L^2$ and $H_0^1$ distances between equivalence classes as
\begin{equation*}
	\begin{aligned}
		\operatorname{dist}\left([U_1],[U_2]\right)&= \inf\limits_{Q\in \mathcal{O}^N} \left\|U_1Q-U_2\right\|, \quad \forall U_1, U_2 \in \big(L^2(\Omega)\big)^N,
		\\\operatorname{dist}_1\left([U_1],[U_2]\right)&= \inf\limits_{Q\in \mathcal{O}^N} \left\|U_1Q-U_2\right\|_1,\quad \forall U_1, U_2 \in \big(H_0^1(\Omega)\big)^N.
	\end{aligned}
\end{equation*}
Given $\delta>0$, for $ U \in \left(H_0^1\left(\Omega\right)\right)^N$ we define a closed $\delta$-neighborhood by
\begin{equation*}
	\begin{aligned}
		B(U,\delta) = \left\{\hat{U}\in \big(H_0^1(\Omega)\big)^N:\|U-\hat{U}\|_1\leqslant \delta\right\},
	\end{aligned}
\end{equation*}
and for $V^{(*)} \in \big(H_0^1(\Omega)\big)^N\bigcap\mathcal{M}^N$, we introduce a closed $\delta$-neighborhood of $[V^{(*)}]$ by
\begin{equation*}
	\begin{aligned}
		B\left([V^{(*)}],\delta\right) = \left\{[\hat{U}]\subset \big(H_0^1(\Omega)\big)^N: \operatorname{dist}_1\left( [\hat{U}],[V^{(*)}]\right)\leqslant \delta\right\}.
	\end{aligned}
\end{equation*}

Let $Y$ denote a real Banach space with norm $\|\cdot\|_Y$. For time-dependent functions, we use the Bochner space
\begin{equation*}
	\begin{aligned}
		L^{p}\left(0, T ; Y\right)&=\left\{U:[0,T)\rightarrow Y\mid 
		\|U\|_{L^{p}\left(0, T ; Y\right)}<\infty	\right\},
	\end{aligned}
\end{equation*}
where $T\in (0,+\infty]$ and the norm $\|U\|_{L^{p}\left(0, T ; Y\right)}$ is given by
\begin{equation*}
	\|U\|_{L^{p}\left(0, T ; Y\right)}=\left\{
	\begin{aligned}
		&	\left(\int_{0}^{T} \|U(\cdot, t)\|_Y^p \mathrm{d}t \right)^{\frac{1}{p}}    &\quad 1\leqslant p<\infty,
		\\&\text{ess}\sup\limits_{0 \leqslant t < T} \|U(\cdot, t)\|_Y & p=\infty.
	\end{aligned} \right. 
\end{equation*}

We denote by
$$
\operatorname{span}\{U\}
=
\left\{
\sum_{i=1}^N a_i u_i:\ a_i\in\mathbb R,\ i=1,\ldots,N
\right\}
$$
the subspace generated by the components of \(U\). 

For symmetric matrices $A,B\in \mathbb{R}^{N\times N}$, we use the positive semidefinite order
$$
A\leqslant B
\quad \Longleftrightarrow \quad
a^{\top}Aa\leqslant a^{\top}Ba,\quad \forall a\in \mathbb{R}^N.
$$
We write $\lambda_{\max}(A)$ and $\lambda_{\min}(A)$ for the largest and smallest eigenvalues of $A$, respectively.

\subsubsection{Preliminary lemma}
\label{Preliminary lemma} 
Let
$$
\mathcal{H}: H_0^1(\Omega) \rightarrow H^{-1}(\Omega)
$$
be the linear operator. We impose the following assumptions:
\begin{itemize}
	\item $\left\langle\mathcal{H}u,v\right\rangle = \left\langle\mathcal{H}v,u\right\rangle$ for all $u, v \in H_0^1(\Omega)$. 
	\item The bilinear form $(u,v) \mapsto \left\langle\mathcal{H}u,v\right\rangle$ is bounded on $H_0^1(\Omega)\times H_0^1(\Omega)$. 
	\item The bilinear form satisfies G{\aa}rding's inequality: there exist constants $c_1>0$ and $c_2\geqslant 0$ such that
	\begin{equation*}
		\begin{aligned}
			\langle\mathcal{H}u,u\rangle
			\geqslant c_1\|\nabla u\|^2-c_2\|u\|^2,
			\quad \forall u\in H_0^1(\Omega).
		\end{aligned}
	\end{equation*}
\end{itemize}

These assumptions hold for many classical linear operators on suitable domains, including second-order elliptic operators. By Section 6.5 of \cite{evans2022partial}, we have the following lemma.
\begin{lemma}\label{lem:eigenfunction basis}
	There exists a nondecreasing sequence of real eigenvalues $\left\{\lambda_i\right\}_{i=1}^{\infty}$ of $\mathcal{H}$ and corresponding eigenfunctions $\left\{v_i\right\}_{i=1}^{\infty} \subset H_0^1(\Omega)$ such that 	$$
	\left\langle \mathcal{H}v_i ,v\right\rangle=\left(\lambda_iv_i ,v\right), \quad \forall v\in H_0^1(\Omega)
	$$ 
	holds for all $ i\in \mathbb{N}_+$. Moreover, the set $\left\{v_i\right\}_{i=1}^{\infty}$ forms an orthonormal basis of $L^2(\Omega)$, i.e., $$\left(v_i, v_j\right)=\delta_{i j},$$ and every $u\in H_0^1(\Omega)$ can be represented as
	\begin{equation*}
		\begin{aligned}
			u = \sum_{i=1}^{\infty}\left(u, v_i\right) v_i.
		\end{aligned}
	\end{equation*}
\end{lemma}
For $N\in \mathbb{N}_+$, we extend $\mathcal{H}$ componentwise to $\big(H_0^1(\Omega)\big)^N$ and use the same notation. Thus, for any $U=\left(u_1,u_2,\cdots,u_N\right)\in\left(H_0^1(\Omega)\right)^N$,
\begin{equation*}
	\begin{aligned}
		\mathcal{H}U =\left(\mathcal{H}u_1,\mathcal{H} u_2, \cdots, \mathcal{H}u_N\right). 
	\end{aligned}
\end{equation*}
\begin{lemma}
	Define the vector 
	$$
	V_{i p}=(0,0, \ldots, 0, \underbrace{v_i}_{p \text {-th position}}, 0, \ldots, 0) \in\left(H_0^1(\Omega)\right)^N 
	$$
	for each $p \in\{1,2, \ldots, N\}$ and $i \geqslant 1$.
	Then, for all $i\in \mathbb{N}_+$ and $1\leqslant p\leqslant N$,
	\begin{equation*}
		\begin{aligned}
			\left\langle\mathcal{H}V_{ip},V\right\rangle
			=\left(\lambda_i V_{ip},V\right),
			\quad \forall V\in \big(H_0^1(\Omega)\big)^N.
		\end{aligned}
	\end{equation*}
	 Moreover, $\left\{V_{i p}\right\}_{i \geqslant 1,1 \leqslant p \leqslant N}$ forms an orthonormal basis of $\left(L^2(\Omega)\right)^N$, that is,
	$$
	\left(V_{i p}, V_{j q}\right)=\delta_{i j} \delta_{p q},
	$$ and it is also a basis of $\big(H_0^1(\Omega)\big)^N$.
	Thus, every $U=\left(u_1, u_2, \cdots, u_N\right) \in\left(H_0^1(\Omega)\right)^N$ can be represented as
	$$
	U=\sum_{i=1}^{\infty} \sum_{p=1}^N\left(U, V_{i p}\right) V_{i p}.
	$$ 
	It is clear that $\left(U, V_{i p}\right)=\left(u_p, v_i\right)$ for $p=1, \ldots, N$.
\end{lemma}

For a fixed $m\in \mathbb{N}_+$, denote the finite-dimensional space $\mathcal{V}_m = \operatorname{span}\{v_1, \cdots, v_m\}$ with the basis $\{v_i \}_{i=1}^{m}$, which also means that 
$$\left\{V_{11}, V_{21}, \cdots, V_{m1}, \cdots\cdots, V_{1N}, V_{2N}, \cdots, V_{mN}\right\}$$
is an orthogonal basis of $\left(\mathcal{V}_m\right)^N \subset \big(H_0^1(\Omega)\big)^N$. Extending the proof of Theorem 30.B in \cite{zeidler2013nonlinear} from $H_0^1(\Omega)$ to $\big(H_0^1(\Omega)\big)^N$ gives the following lemma.
\begin{lemma}
	Define the operator $\mathcal{P}_m: \big(H^{-1}(\Omega)\big)^N \rightarrow  \left(\mathcal{V}_m\right)^N$ by
	\begin{equation*}
		\begin{aligned}
			\mathcal{P}_m \mathcal{F} = \sum\limits_{i=1}^m \sum\limits_{p=1}^N  \langle \mathcal{F}, V_{ip}\rangle V_{ip}, \quad \forall \mathcal{F} \in \big(H^{-1}(\Omega)\big)^N.
		\end{aligned}
	\end{equation*}
	Then $\mathcal{P}_m$ is linear and continuous and satisfies
	\begin{equation*}
		\begin{aligned}
			\langle  \mathcal{P}_m \mathcal{F}, U\rangle = \langle \mathcal{F}, \mathcal{P}_m U\rangle \text{ for all } \mathcal{F} \in \big(H^{-1}(\Omega)\big)^N \text{ and }U \in \big(H_0^1(\Omega)\big)^N.
		\end{aligned}
	\end{equation*} Moreover, its restriction to $\big(L^2(\Omega)\big)^N$ can be written as
	\begin{equation*}
		\begin{aligned}
			\mathcal{P}_m U = \sum\limits_{i=1}^m  \sum\limits_{p=1}^N   \left(  U, V_{ip} \right) V_{ip}, \quad \forall U \in \left(L^2\left(\Omega\right)\right)^N.
		\end{aligned}
	\end{equation*}
Thus, $\mathcal{P}_m|_{\big(L^2(\Omega)\big)^N}$ is the orthogonal projection from $\big(L^2(\Omega)\big)^N$ onto $\left(\mathcal{V}_m\right)^N$.
\end{lemma}

Finally, we introduce a lemma, which can be obtained from Lemma 2.2 of \cite{dai2017conjugate} together with a standard analysis.
\begin{lemma}
	Let $U=\left(u_1, u_2, \cdots, u_N\right) \in\big(L^2(\Omega)\big)^N$, $V=\left(v_1, v_2, \cdots, v_N\right) \in\big(L^2(\Omega)\big)^N$ and $A\in \mathbb{R}^{N\times N}$. Then
	\begin{equation*}
		\begin{aligned}
			&	\| U^\top V \| \leqslant \|U\|\|V\|,
			\\&	\| U A \| \leqslant \|U\|\|A\|,
			\\&	\| U A \| \leqslant \sigma_{\max}(U)\|A\|,
		\end{aligned}
	\end{equation*}
	where $\sigma_{\max}(U) = \sqrt{\lambda_{\max}\left( U^\top U\right)}$. It implies that the norms for functions and matrices are compatible.
\end{lemma}

\subsection{Eigenvalue problem}
We now consider the eigenvalue problem for the first $N$ smallest eigenvalues of the operator $\mathcal{H}$:
\begin{equation*}
	\left\{\begin{aligned}
		&\left\langle \mathcal{H}v_j,v\right\rangle
		=\left(\lambda_j v_j,v\right),
		\quad \forall v\in H_0^1(\Omega),\quad 1\leqslant j\leqslant N,\\
		&(v_i,v_j)=\delta_{ij},\quad 1\leqslant i,j\leqslant N,
	\end{aligned}\right.
\end{equation*}
which is equivalent to
\begin{equation}
	\label{linear eigenvalue problem}
	\left\{ \begin{aligned} 
		&\left\langle \mathcal{H} V^{(*)}, V\right\rangle=\left( V^{(*)}\Lambda,  V\right),   \quad \forall V\in \big(H_0^1(\Omega)\big)^N,
		\\&(V^{(*)})^\top V^{(*)}= I_N,
	\end{aligned}\right.
\end{equation} 
where $\Lambda =\operatorname{diag}(\lambda_1,\lambda_2,\cdots, \lambda_N)$ with $\lambda_1\leqslant \lambda_2\leqslant \cdots\leqslant \lambda_N    \ (<\lambda_{N+1})   $ being the $N$ smallest eigenvalues of $\mathcal{H}$, and the columns of $V^{(*)}=\left(v_1, v_2, \ldots, v_N\right) \in\left(H_0^1\left(\Omega\right)\right)^N$ are the corresponding eigenfunctions.

Since the eigenspaces are unchanged under a scalar shift of the operator, we replace the original operator by a shifted operator so that its first $N$ eigenvalues are negative. For simplicity, the shifted operator is still denoted by $\mathcal H$. Thus, throughout the paper, we may assume that $\mathcal H$ has at least $N$ negative eigenvalues.

Associated with the operator $\mathcal{H}$, we define an energy  functional
	$$ E(W)=\frac{1}{2} \operatorname{tr}\left(W^\top \mathcal{H} W\right) \quad \forall W\in (H^1_0(\Omega))^N,$$
	where the energy functional $E$ is invariant under orthogonal transformations, i.e., 
	$$E(W)=E(W Q),\quad \forall Q \in \mathcal{O}^N,\  W \in \big(H_0^1(\Omega)\big)^N. $$
The eigenvalue problem (\ref{linear eigenvalue problem}) can also be formulated as the minimization problem
\begin{equation}
	\label{minimization problem}
	\begin{aligned}
		\min\limits_{ W\in \mathcal{M}^N} E(W).
	\end{aligned}
\end{equation}
To get rid of the nonuniqueness of the minimizer caused by this invariance, we consider (\ref{minimization problem}) on the Grassmann manifold $\mathcal{G}^N$. Actually, (\ref{minimization problem}) is equivalent to 
\begin{equation}
	\label{minimization in Grassmann}
	\min_{[W] \in \mathcal{G}^N} E (W),
\end{equation} 
and $[V^{(*)}]$ is the unique minimizer of (\ref{minimization in Grassmann}) by the spectral gap $\lambda_{N+1}- \lambda_N > 0$ (see \cite{schneider2009direct}).

The gradient on the Grassmann manifold $\mathcal{G}^N$ of $E(U)$ at $[U]$ is 
\begin{equation*}
	\begin{aligned}
		\nabla_G E(U) = \nabla E(U) - U U^\top\nabla E(U), 
	\end{aligned}
\end{equation*}
where 
\begin{equation*}
	\begin{aligned}
		\nabla E(U) = \mathcal{H}U
	\end{aligned}
\end{equation*}
is the gradient of $E$.

\section{A quasi-Grassmannian gradient flow model}
We extend the Grassmann gradient $\nabla_GE$, originally defined on $\mathcal{G}^N$, to the full space $\big(H_0^1(\Omega)\big)^N$ by the same expression and keep the notation $\nabla_GE$.
With this extension, we introduce the evolution equation
\begin{equation}
	\label{operator of gradient flow based model}
	\left\{\begin{aligned}
		& \frac{\mathrm{d}  U  }{\mathrm{d}  t} =-\nabla_G E(U) &&\text{on } \Omega \times [0,\infty),
		\\&U(0)=U_0           &&\text{on }\Omega\\
	\end{aligned}\right.
\end{equation}
to address the eigenvalue problem (\ref{linear eigenvalue problem}). We refer to this system as the quasi-Grassmannian gradient flow model. The evolution is driven by the negative gradient $-\nabla_G E(U) $. As detailed in the following section, the asymptotic behavior of the solution $[U(t)]$ indicates that, as time progresses, the solution becomes ``quasi-confined" to the Grassmann manifold $\mathcal{G}^N$; in other words, $U(t)$ exhibits asymptotic orthogonality.

In the following, we will investigate the properties of the solution of (\ref{operator of gradient flow based model}) in detail.
\subsection{Well-posedness}
This subsection discusses the well-posedness of the weak formulation of (\ref{operator of gradient flow based model}).
The weak formulation of (\ref{operator of gradient flow based model}) is
\begin{equation}
	\label{refined weak formulation}
	\begin{aligned}
		\left\{
		\begin{aligned}
			& \frac{\mathrm{d}    }{\mathrm{d}  t}  \left( U(t), V\right)+ \left\langle \nabla_GE(U(t)), V \right\rangle  = 0, \quad \forall V \in \big(H_0^1(\Omega)\big)^N,
			\\&U(0)  =U_0 \in  \big(H_0^1(\Omega)\big)^N,
		\end{aligned}
		\right.
	\end{aligned}
\end{equation}
where $\frac{\mathrm{d}}{\mathrm{d}t}$ is to be understood as the generalized derivative \cite{zeidler1989nonlinear}, namely, for any $T>0$,
\begin{equation*}
	\begin{aligned}
		\int_{0}^{T}  \frac{\mathrm{d}}{\mathrm{d}t}   \left( U(t), V\right)   \phi(t)  \mathrm{d}t 	 = 	-\int_{0}^{T}   \left( U(t), V\right)\frac{\mathrm{d}\phi(t) }{\mathrm{d}t}\mathrm{d}t ,\quad \forall \phi \in C_0^{\infty}(0,T).
	\end{aligned}
\end{equation*}

The assumptions in Section~\ref{Preliminary lemma} imply the following estimates:
\begin{itemize}
	\item The bilinear form induced by $\mathcal{H}$ is bounded on $\big(H_0^1(\Omega)\big)^N\times\big(H_0^1(\Omega)\big)^N$. Hence $\left\langle\mathcal{H}U,V\right\rangle = \operatorname{tr}\left(V^\top \mathcal{H}U\right)$ is bounded. That is, there exists a constant $C_b>0$ such that
	\begin{equation*}
		\begin{aligned}
			|\left\langle\mathcal{H}U,V\right\rangle | \leqslant C_b \|U\|_1\|V\|_1 ,\quad \forall U,V \in \big(H_0^1(\Omega)\big)^N.
		\end{aligned}
	\end{equation*}
	\item There holds G\aa rding's inequality 
	\begin{equation*}
		\begin{aligned}
			\left\langle\mathcal{H}U,U\right\rangle  \geqslant c_1\|\nabla U\|^2 -c_2\|U\|^2 ,\quad \forall U \in  \big(H_0^1(\Omega)\big)^N.
		\end{aligned}
	\end{equation*}
	Here, $c_1>0$ and $c_2\geqslant 0$ are the same as those in Section \ref{Preliminary lemma}.
\end{itemize}
These results ensure the well-posedness of the quasi-Grassmannian gradient flow model.

To address the well-posedness, we first introduce the finite-dimensional approximation of (\ref{operator of gradient flow based model}):
finding the solution $U^{(m)}: [0,T)\rightarrow \left(\mathcal{V}_m\right)^N$ for any given $T>0$ of the following equation
\begin{equation}
	\label{Galerkin equation}
	\left\{\begin{aligned}
		&\frac{\mathrm{d}    }{\mathrm{d}  t} U^{(m)}(t)+ \mathcal{P}_m\nabla_GE(U^{(m)}(t)) = 0,
		\\&U^{(m)}(0) = \mathcal{P}_mU_0,
	\end{aligned}\right.
\end{equation}
and establish the existence and uniqueness of solutions of (\ref{Galerkin equation}) on $[0,T)$.

	We define the admissible set as
	\begin{equation*}
		\begin{aligned}
			\mathcal{L}_0 = \left\{U\in \big(H_0^1(\Omega)\big)^N: U^\top \mathcal{H}U\leqslant0\right\}.
		\end{aligned}
	\end{equation*}
	By choosing the scalar shift above sufficiently large, we assume that the initial value satisfies
	$$
	U_0^\top\mathcal HU_0\leqslant0,
	$$
	or equivalently, $U_0\in\mathcal{L}_0$.
	
	\begin{lemma}
		\label{lemma: UHU<0}
		If $U_0 \in \mathcal{L}_0$, then for any solution $ U^{(m)}(t)$ of (\ref{Galerkin equation}), there holds
		\begin{equation*}
			U^{(m)}(t) \in \mathcal{L}_0,\quad \forall t\geqslant0.
		\end{equation*}
	\end{lemma}
	\begin{proof}
		The proof has been deferred to Appendix~\ref{proof of lemma: UHU<0}.
	\end{proof}
	Lemma~\ref{lemma: UHU<0} shows that the admissible set $\mathcal{L}_0$ is preserved along the gradient flow.

\begin{proposition}
	\label{existence of weak formulation of Galerkin equation}
	If $U_0\in 	\mathcal{L}_0 $, then there exists a time $T_m^*>0$ such that (\ref{Galerkin equation}) has a unique solution
	\begin{equation*}
		\begin{aligned}
			U^{(m)} \in L^{2}\left(0, T_m^* ; \big(H_0^1(\Omega)\big)^N\right)  \text{ with } \frac{\mathrm{d}  U^{(m)} }{\mathrm{d}  t} \in L^{2}\left(0, T_m^* ; \big(H^{-1}(\Omega)\big)^N\right).
		\end{aligned}
	\end{equation*}
\end{proposition}
\begin{proof}
	The proof has been provided in Appendix~\ref{proof of proposition:existence of weak formulation of Galerkin equation}.
\end{proof}

The well-posedness of (\ref{Galerkin equation}) and regularity of $	U^{(m)}(t)$ on $[0,T_m^*)$, as established in Proposition \ref{existence of weak formulation of Galerkin equation}, can actually be extended to any finite time interval $t\in [0, T)$ for any given $T> 0$. This means that the unique solution $U^{(m)}(t)$ is globally defined with its regularity properties maintained over the entire interval $[0,T)$. The proof is given in Appendix~\ref{proof of corollary: T be infty}.
\begin{corollary}
	\label{T be infty}
	If $U_0\in 	\mathcal{L}_0 $, then for any given $T>0$, (\ref{Galerkin equation}) has a unique solution
	\begin{equation*}
		\begin{aligned}
			U^{(m)} \in C\left([0, T) ;\big(H_0^1(\Omega)\big)^N\right).
		\end{aligned}
	\end{equation*}
\end{corollary}

Given $T>0$, we see from Proposition 23.20 in \cite{zeidler1989nonlinear} that 
\begin{lemma}
	\label{lemma of du}
	Suppose $U \in L^2\left(0, T ; \big(H_0^1(\Omega)\big)^N\right)$. If there exists a function $\mathcal{W} \in L^2\left(0, T ; \big(H^{-1}(\Omega)\big)^N \right)$ satisfying
	\begin{equation*}
		\begin{aligned}
			\int_{0}^{T}  \frac{\mathrm{d}}{\mathrm{d}t}   \left( U(t), V\right)   \phi(t)  \mathrm{d}t =  \int_0^T\langle \mathcal{W}(t), V\rangle \phi(t) d t, \quad \forall V \in  \big(H_0^1(\Omega)\big)^N, \phi \in C_0^{\infty}(0, T),
		\end{aligned}
	\end{equation*}
then the generalized derivative $\frac{dU}{dt} =\mathcal{W}$ exists and satisfies
	\begin{equation*}
		\begin{aligned}
			\frac{\mathrm{d} U   }{\mathrm{d}  t} \in L^2\left(0, T ;  \big(H^{-1}(\Omega)\big)^N\right),
		\end{aligned}
	\end{equation*}
	and
	\begin{equation*}
		\begin{aligned}
			\frac{\mathrm{d}    }{\mathrm{d}  t}(U, V)=\left\langle  \frac{\mathrm{d}  U  }{\mathrm{d}  t}, V\right\rangle, \quad \forall V \in  \big(H_0^1(\Omega)\big)^N.
		\end{aligned}
	\end{equation*}
\end{lemma}

Let
$$
\mathcal{D}(\mathcal{H})=\left\{u\in H_0^1(\Omega):\mathcal{H}u\in L^2(\Omega)\right\}
$$
be equipped with the graph norm
$$
\|u\|_{\mathcal{D}(\mathcal{H})}=\|\mathcal{H}u\|+\|u\|_1.
$$
\begin{lemma}\label{lem:compact embedding}
	The embedding $\mathcal{D}(\mathcal{H})\hookrightarrow H_0^1(\Omega)$ is compact. Consequently, the product embedding $(\mathcal{D}(\mathcal{H}))^N\hookrightarrow (H_0^1(\Omega))^N$ is compact.
\end{lemma}
\begin{proof}
	The proof is given in Appendix~\ref{proof of lemma: compact embedding}.
\end{proof}

We then state and prove the following existence and uniqueness result for (\ref{operator of gradient flow based model}).
\begin{theorem}
	\label{Existence of weak solutions}
	If $U_0\in\mathcal L_0$, then for any given $T>0$, \eqref{operator of gradient flow based model} has a unique weak solution
	\begin{equation*}
		\begin{aligned}
			U\in L^2\left(0,T;(\mathcal{D}(\mathcal{H}))^N\right)\cap L^\infty\left(0,T;(H_0^1(\Omega))^N\right),
			\text{ with }
			\frac{\mathrm dU}{\mathrm dt}\in L^2\left(0,T;(H^{-1}(\Omega))^N\right).
		\end{aligned}
	\end{equation*}
\end{theorem}
\begin{proof}
	Let $\mathcal V_m=\operatorname{span}\{v_1,\ldots,v_m\}$. Throughout the proof, $C_T$ denotes a positive constant which may change from line to line but is independent of $m$. By Proposition~\ref{existence of weak formulation of Galerkin equation} and Corollary~\ref{T be infty}, \eqref{Galerkin equation} admits a global solution $U^{(m)}$. Moreover, by \eqref{eq: uniform bound of Um} and the estimate for $\frac{\mathrm dU^{(m)}}{\mathrm dt}$ in Proposition~\ref{existence of weak formulation of Galerkin equation}, for every $T>0$,
	\begin{equation*}
		\begin{aligned}
			\sup_{0\leqslant t\leqslant T}\|U^{(m)}(t)\|^2	+	\int_0^T\|U^{(m)}(t)\|_1^2\,\mathrm{d}t	+	\int_0^T	\left\|	\frac{\mathrm{d}U^{(m)}}{\mathrm{d}t}(t)	\right\|_{-1}^2	\,\mathrm{d}t	\leqslant C_T.
		\end{aligned}
	\end{equation*}

	By Lemma~\ref{lemma: UHU<0},
	$$
	\big(U^{(m)}(t)\big)^\top\mathcal H U^{(m)}(t)\leqslant0,
	\qquad 0\leqslant t\leqslant T .
	$$
	Hence, by the G{\aa}rding inequality,
	\begin{equation*}
			0\geqslant	\left\langle	\mathcal H U^{(m)}(t),U^{(m)}(t)	\right\rangle \geqslant	c_1\|\nabla U^{(m)}(t)\|^2-c_2\|U^{(m)}(t)\|^2 .
	\end{equation*}
Together with the uniform $L^\infty(0,T;(L^2(\Omega))^N)$ bound obtained from \eqref{eq: uniform bound of Um}, this gives
	$$
	\{U^{(m)}\}_{m\geqslant1}	\quad\text{is bounded in }	L^\infty\left(0,T;(H_0^1(\Omega))^N\right).
	$$

	Since $U^{(m)}(t)\in(\mathcal V_m)^N$, we have $\mathcal H U^{(m)}(t)\in(\mathcal V_m)^N$. Taking $V=\mathcal H U^{(m)}$ in the Galerkin equation yields
	\begin{equation*}
		\begin{aligned}
			\left\langle
			\frac{\mathrm dU^{(m)}}{\mathrm dt},
			\mathcal H U^{(m)}
			\right\rangle
			+
			\|\mathcal H U^{(m)}\|^2
			-
			\operatorname{tr}
			\left(
			\big(\mathcal H U^{(m)}\big)^\top
			U^{(m)}
			\big(U^{(m)}\big)^\top
			\mathcal H U^{(m)}
			\right)
			=0 .
		\end{aligned}
	\end{equation*}
	Using the self-adjointness of \(\mathcal H\) on \((\mathcal V_m)^N\), we obtain
	\begin{equation} \label{eq: relation of HU}
		\begin{aligned}
			\frac12
			\frac{\mathrm d}{\mathrm dt}
			\operatorname{tr}
			\left(
			\big(U^{(m)}\big)^\top
			\mathcal H U^{(m)}
			\right)
			+
			\|\mathcal H U^{(m)}\|^2
			=
			\operatorname{tr}
			\left[
			\left(
			\big(U^{(m)}\big)^\top
			\mathcal H U^{(m)}
			\right)^2
			\right].
		\end{aligned}
	\end{equation}
	By the spectral lower bound for $\mathcal H$ and Lemma~\ref{lemma: UHU<0},
	$$
	\lambda_1\big(U^{(m)}(t)\big)^\top U^{(m)}(t)
	\leqslant
	\big(U^{(m)}(t)\big)^\top\mathcal H U^{(m)}(t)
	\leqslant0 ,
	$$
	we have
	\begin{equation*}
		\begin{aligned}
			\int_0^T
			\operatorname{tr}
			\left[
			\left(
			\big(U^{(m)}\big)^\top
			\mathcal H U^{(m)}
			\right)^2
			\right]
			\,\mathrm dt
			\leqslant C_T .
		\end{aligned}
	\end{equation*}
Integrating \eqref{eq: relation of HU} over $(0,T)$, we arrive at
	$$
	\int_0^T\|\mathcal H U^{(m)}(t)\|^2\,\mathrm dt
	\leqslant C_T .
	$$
	Together with the uniform $L^2(0,T;(H_0^1(\Omega))^N)$ bound, this gives
	$$
	\{U^{(m)}\}_{m\geqslant1}
	\quad\text{is bounded in }
	L^2\left(0,T;(\mathcal D(\mathcal H))^N\right).
	$$

	By Lemma~\ref{lem:compact embedding}, the embedding
	\((\mathcal D(\mathcal H))^N\hookrightarrow(H_0^1(\Omega))^N\) is compact. The Aubin--Lions--Simon lemma in Section~8 of \cite{simon1986compact} yields a subsequence, still denoted by \(\{U^{(m)}\}\), and a function \(U\) such that
	\begin{equation*}
		\begin{aligned}
			U^{(m)}
			&\to U
			&&\text{strongly in }
			L^2\left(0,T;(H_0^1(\Omega))^N\right),
			\\
			U^{(m)}
			&\rightharpoonup U
			&&\text{weakly in }
			L^2\left(0,T;(\mathcal D(\mathcal H))^N\right),
			\\
			\mathcal H U^{(m)}
			&\rightharpoonup \mathcal H U
			&&\text{weakly in }
			L^2\left(0,T;(L^2(\Omega))^N\right).
		\end{aligned}
	\end{equation*}
	Furthermore, after extracting a further subsequence and retaining the same notation, we have
	$$
	U^{(m)}\rightharpoonup^\ast U
	\quad\text{weakly* in }
	L^\infty\left(0,T;(H_0^1(\Omega))^N\right).
	$$
	Consequently,
	$$
	U\in
	L^2\left(0,T;(\mathcal D(\mathcal H))^N\right)
	\cap
	L^\infty\left(0,T;(H_0^1(\Omega))^N\right).
	$$
	
	We next pass to the limit in the weak formulation. 
	 Since $U^{(m)}\to U$ strongly in $L^2(0,T;(L^2(\Omega))^N)$ and $\{U^{(m)}\}_{m\geqslant1}$ is uniformly bounded in $L^\infty(0,T;(L^2(\Omega))^N)$, for every fixed $V\in(L^2(\Omega))^N$ we have
		$$
		U^{(m)}\big((U^{(m)})^\top V\big)\to U(U^\top V)\quad\text{strongly in }L^2\left(0,T;(L^2(\Omega))^N\right).
		$$
		Indeed,
		\begin{equation*}
			\begin{aligned}
			&	\left\|U^{(m)}\big((U^{(m)})^\top V\big)-U(U^\top V)\right\|_{L^2(0,T;(L^2)^N)}
			\\&\leqslant C\|V\|\left(\|U^{(m)}\|_{L^\infty(0,T;(L^2)^N)}+\|U\|_{L^\infty(0,T;(L^2)^N)}\right)\|U^{(m)}-U\|_{L^2(0,T;(L^2)^N)}.
			\end{aligned}
		\end{equation*}
	Therefore, by the weak convergence of \(\mathcal H U^{(m)}\),
	\begin{equation*}
		\begin{aligned}
			\int_0^T
			\left\langle
			\mathcal H U^{(m)},
			U^{(m)}\big((U^{(m)})^\top V\big)
			\right\rangle
			\phi(t)\,\mathrm dt
		\to
			\int_0^T
			\left\langle
			\mathcal H U,
			U(U^\top V)
			\right\rangle
			\phi(t)\,\mathrm dt .
		\end{aligned}
	\end{equation*}
	Also,
	$$
	\int_0^T
	\langle\mathcal H U^{(m)},V\rangle
	\phi(t)\,\mathrm dt
	\to
	\int_0^T
	\langle\mathcal H U,V\rangle
	\phi(t)\,\mathrm dt .
	$$
	Hence
	\begin{equation*}
		\begin{aligned}
			\int_0^T
			\left\langle
			\nabla_GE(U^{(m)}(t)),V
			\right\rangle
			\phi(t)\,\mathrm dt
			\to
			\int_0^T
			\left\langle
			\nabla_GE(U(t)),V
			\right\rangle
			\phi(t)\,\mathrm dt .
		\end{aligned}
	\end{equation*}
	
	Let \(V\in(\mathcal V_k)^N\) and \(m\geqslant k\). From the Galerkin equation,
	$$
	\frac{\mathrm d}{\mathrm dt}(U^{(m)}(t),V)
	=
	-
	\left\langle
	\nabla_GE(U^{(m)}(t)),V
	\right\rangle .
	$$
	Multiplying by \(\phi\in C_0^\infty(0,T)\), integrating over \((0,T)\), and passing to the limit gives
	\begin{equation*}
		\begin{aligned}
			-\int_0^T (U(t),V)\phi'(t)\,\mathrm dt
			+
			\int_0^T
			\left\langle
			\nabla_GE(U(t)),V
			\right\rangle
			\phi(t)\,\mathrm dt
			=0 .
		\end{aligned}
	\end{equation*}

	Since \(\bigcup_{k\geqslant1}(\mathcal V_k)^N\) is dense in
	\((H_0^1(\Omega))^N\), the same identity holds for every
	\(V\in(H_0^1(\Omega))^N\). By Lemma~\ref{lemma of du}, there exists the generalized derivative
	$$
	\frac{\mathrm dU}{\mathrm dt}
	=
	-\nabla_GE(U)
	\quad\text{in }
	L^2\left(0,T;(H^{-1}(\Omega))^N\right).
	$$
	Thus \(U\) satisfies
	$$
	\left\langle
	\frac{\mathrm dU}{\mathrm dt}(t),V
	\right\rangle
	+
	\left\langle
	\nabla_GE(U(t)),V
	\right\rangle
	=
	0,
	\qquad
	\forall V\in(H_0^1(\Omega))^N.
	$$
Note that
	$$
	U\in L^2\left(0,T;(H_0^1(\Omega))^N\right),
	\qquad
	\frac{\mathrm dU}{\mathrm dt}
	\in
	L^2\left(0,T;(H^{-1}(\Omega))^N\right),
	$$
	we obtain from the Lions--Magenes theorem (see Section~3, Theorem~3.1 of \cite{lionsmagenes1972}) that
	$$
	U\in C\left([0,T];(L^2(\Omega))^N\right).
	$$

	It remains to prove uniqueness. Let $U_1$ and $U_2$ be two weak solutions with the same initial data $U_0$, and set
	$$
	W=U_1-U_2 .
	$$
	Subtracting the two weak formulations and testing with $W$ yields
	$$
	\frac12\frac{\mathrm d}{\mathrm dt}\|W\|^2+\langle\mathcal H W,W\rangle=\operatorname{tr}\left(W^\top\left(U_1(U_1^\top\mathcal H U_1)-U_2(U_2^\top\mathcal H U_2)\right)\right).
	$$
	Since
	$$
	U_1(U_1^\top\mathcal H U_1)-U_2(U_2^\top\mathcal H U_2)=W(U_1^\top\mathcal H U_1)+U_2\left(U_1^\top\mathcal H W+W^\top\mathcal H U_2\right),
	$$
	the boundedness of the bilinear form induced by $\mathcal H$ gives
	\begin{equation*}
		\begin{aligned}
			\|U_i^\top\mathcal H U_i\| &\leqslant C\|U_i\|_1^2,\qquad i=1,2,\\
			\|U_1^\top\mathcal H U_1-U_2^\top\mathcal H U_2\| &\leqslant C(\|U_1\|_1+\|U_2\|_1)\|W\|_1 .
		\end{aligned}
	\end{equation*}
	Consequently,
	\begin{equation*}
		\begin{aligned}
		&	\left|\operatorname{tr}\left(W^\top\left(U_1(U_1^\top\mathcal H U_1)-U_2(U_2^\top\mathcal H U_2)\right)\right)\right| 
			\\&\leqslant C\|U_1\|_1^2\|W\|^2+C\|U_2\|(\|U_1\|_1+\|U_2\|_1)\|W\|\|W\|_1\\
			&\leqslant \varepsilon\|W\|_1^2+C_\varepsilon\left(1+\|U_1\|_1^2+\|U_2\|_1^2\right)^2\|W\|^2 .
		\end{aligned}
	\end{equation*}
	By the G{\aa}rding inequality,
	$$
	\langle\mathcal H W,W\rangle\geqslant c_1\|W\|_1^2-(c_1+c_2)\|W\|^2 .
	$$
	Choosing $\varepsilon=c_1/2$, we obtain
	$$
	\frac{\mathrm d}{\mathrm dt}\|W(t)\|^2\leqslant g(t)\|W(t)\|^2,
	$$
	 where
	$$
	g(t)=2\left(c_1+c_2+C_{c_1/2}\left(1+\|U_1(t)\|_1^2+\|U_2(t)\|_1^2\right)^2\right).
	$$
	Since $U_1,U_2\in L^\infty(0,T;(H_0^1(\Omega))^N)$, we have $g\in L^\infty(0,T)\subset L^1(0,T)$.
	 Since $W(0)=0$ and $W\in C([0,T];(L^2(\Omega))^N)$, Gr\"{o}nwall's inequality gives
	$$
	\|W(t)\|^2=0,\qquad 0\leqslant t\leqslant T .
	$$
	Thus $U_1=U_2$, and the weak solution is unique.
\end{proof}


\subsection{Analytic representation}
\label{subsec:analytic representation}

In this subsection, we derive an analytic representation of the weak solution.
We first recall the following elementary linear evolution result.

\begin{lemma}
	\label{solution of linear evolution problem}
	Let \(k>0\). For every \(U_0\in(\mathcal D(\mathcal H))^N\), the linear evolution problem
	\begin{equation}
		\label{linear evolution problem}
		\left\{
		\begin{aligned}
			&\frac{\mathrm d W}{\mathrm dt}=-k\mathcal HW,
			&& t>0,\\
			&W(0)=U_0
		\end{aligned}
		\right.
	\end{equation}
	has a unique solution
	$$
	W\in
	C^1\big([0,\infty);(L^2(\Omega))^N\big)
	\cap
	C\big([0,\infty);(\mathcal D(\mathcal H))^N\big),
	$$
	given explicitly by
	$$
	W(t)=\exp(-k\mathcal Ht)U_0.
	$$
\end{lemma}

\begin{proof}
	The proof is given in Appendix~\ref{proof of lemma:solution of linear evolution problem}.
\end{proof}

\begin{lemma}
	\label{negative condition of initial value}
	Let \(k>0\). If \(U_0\in(\mathcal D(\mathcal H))^N\) satisfies
	$$
	U_0^\top\mathcal H U_0\leqslant 0,
	$$
	then
	$$
	U_0^\top e^{-k\mathcal Ht}U_0\geqslant U_0^\top U_0,
	\qquad t\geqslant0.
	$$
\end{lemma}

\begin{proof}
	The proof is given in Appendix~\ref{proof of lemma: negative condition of initial value}.
\end{proof}

We next introduce the representation of the Galerkin solution.

\begin{lemma}
	\label{expression of Galerkin solution}
	Assume that \(U_0\in\mathcal L_0\). Then the solution of the Galerkin equation \eqref{Galerkin equation} can be expressed as
	$$
	\begin{aligned}
		U^{(m)}(t)
		=&\,
		\exp(-\mathcal H_m t)U^{(m)}(0)
		\left[
		I_N-\big(U^{(m)}(0)\big)^\top U^{(m)}(0)
		\right.
		\\
		&\qquad\qquad\left.
		+
		\big(U^{(m)}(0)\big)^\top
		\exp(-2\mathcal H_m t)U^{(m)}(0)
		\right]^{-1/2}
		Q^{(m)}(t),
	\end{aligned}
	$$
	where \(Q^{(m)}(t)\in\mathcal O^N\) is a time-dependent orthogonal matrix.
\end{lemma}

\begin{proof}
	The proof is given in Appendix~\ref{proof of lemma: expression of Galerkin solution}.
\end{proof}

We now pass to the weak solution of the infinite-dimensional problem.

\begin{theorem}
	\label{expression of solution(linear operator)}
	Assume that $U_0\in\mathcal L_0\cap(\mathcal D(\mathcal H))^N$.
	Let \(U(t)\) be the solution of \eqref{operator of gradient flow based model}. Then there exists a time-dependent orthogonal matrix
	\(Q(t)\in\mathcal O^N\) such that
	\begin{equation}
		\label{expression of solution(constant)}
		U(t)
		=
		\exp(-\mathcal Ht)U_0
		\left[
		I_N-U_0^\top U_0
		+
		U_0^\top\exp(-2\mathcal Ht)U_0
		\right]^{-1/2}
		Q(t).
	\end{equation}
\end{theorem}

\begin{proof}
	Due to $U_0\in\mathcal{L}_0$, Lemma~\ref{negative condition of initial value} tells
	$$
	U_0^\top\exp(-2\mathcal Ht)U_0
	\geqslant
	U_0^\top U_0,
	\qquad t\geqslant0 .
	$$
	Hence
	$$
	I_N-U_0^\top U_0
	+
	U_0^\top\exp(-2\mathcal Ht)U_0
	\geqslant I_N .
	$$
	Thus the matrix
	$$
	I_N-U_0^\top U_0
	+
	U_0^\top\exp(-2\mathcal Ht)U_0
	$$
	is positive definite, and its inverse square root is well defined.

	Let \(\mathcal H_m\) denote the restriction of \(\mathcal H\) to \((\mathcal V_m)^N\).
	Since \(U^{(m)}(0)\in(\mathcal V_m)^N\),
	$$
	\exp(-\mathcal H_m t)U^{(m)}(0)
	=
	\exp(-\mathcal Ht)U^{(m)}(0).
	$$
		For each fixed $t\geqslant0$, the semigroup bound leads to
	\begin{equation*}
		\begin{aligned}
			\left\|
			\exp(-k\mathcal Ht)U_0^{(m)}
			-
			\exp(-k\mathcal Ht)U_0
			\right\|
			&\leqslant
			e^{-k\lambda_1t}
			\|U_0^{(m)}-U_0\|.
		\end{aligned}
	\end{equation*}
	Hence
	$$
	\exp(-k\mathcal Ht)U_0^{(m)}
	\to
	\exp(-k\mathcal Ht)U_0
	\qquad\text{in }\big(L^2(\Omega)\big)^N.
	$$
	Consequently,
	$$
	\begin{aligned}
		&I_N-\big(U^{(m)}(0)\big)^\top U^{(m)}(0)
		+
		\big(U^{(m)}(0)\big)^\top
		\exp(-2\mathcal H_m t)U^{(m)}(0)
		\\
		&\qquad\longrightarrow
		I_N-U_0^\top U_0
		+
		U_0^\top\exp(-2\mathcal Ht)U_0,
	\end{aligned}
	$$
which yields
	$$
	\begin{aligned}
		&\left[
		I_N-\big(U^{(m)}(0)\big)^\top U^{(m)}(0)
		+
		\big(U^{(m)}(0)\big)^\top
		\exp(-2\mathcal H_m t)U^{(m)}(0)
		\right]^{-1}
		\\
		&\qquad\longrightarrow
		\left[
		I_N-U_0^\top U_0
		+
		U_0^\top\exp(-2\mathcal Ht)U_0
		\right]^{-1}.
	\end{aligned}
	$$
	
	By Lemma~\ref{expression of Galerkin solution},
	$$
	\begin{aligned}
		U^{(m)}(t)\big(U^{(m)}(t)\big)^\top
		=&\,
		\exp(-\mathcal H_m t)U^{(m)}(0)
		\left[
		I_N-\big(U^{(m)}(0)\big)^\top U^{(m)}(0)
		\right.
		\\
		&\left.
		+
		\big(U^{(m)}(0)\big)^\top
		\exp(-2\mathcal H_m t)U^{(m)}(0)
		\right]^{-1}
		\big(U^{(m)}(0)\big)^\top
		\exp(-\mathcal H_m t).
	\end{aligned}
	$$
	Passing to the limit gives
	$$
	\begin{aligned}
		U^{(m)}(t)\big(U^{(m)}(t)\big)^\top
		\to&
		\exp(-\mathcal Ht)U_0
		\left[
		I_N-U_0^\top U_0
		+
		U_0^\top\exp(-2\mathcal Ht)U_0
		\right]^{-1}
		\\
		&\quad
		U_0^\top\exp(-\mathcal Ht).
	\end{aligned}
	$$
	
	On the other hand, by Theorem~\ref{Existence of weak solutions}, after passing to a subsequence,
	$$
	U^{(m)}
	\to
	U
	\qquad
	\text{strongly in }
	L^2(0,T;(L^2(\Omega))^N).
	$$
	Hence, after extracting a further subsequence, still denoted by
	$\{U^{(m)}\}$, we have
	$$
	U^{(m)}(t)
	\to
	U(t)
	\qquad
	\text{strongly in }(L^2(\Omega))^N
	$$
	for a.e. \(t\in(0,T)\). Therefore,
	$$
	\begin{aligned}
		U(t)U(t)^\top
		=
		\exp(-\mathcal Ht)U_0
		\left[
		I_N-U_0^\top U_0
		+
		U_0^\top\exp(-2\mathcal Ht)U_0
		\right]^{-1}
		U_0^\top\exp(-\mathcal Ht)
	\end{aligned}
	$$
	for a.e. \(t\in(0,T)\). Since both sides are continuous in \(t\), the identity extends to every \(t\in[0,T]\).
	As \(T>0\) is arbitrary, it holds for all \(t\geqslant0\).
	
	We see that $U(t)U(t)^\top$ can be written as
	$$
	\begin{aligned}
		&\exp(-\mathcal Ht)U_0
		\left[
		I_N-U_0^\top U_0
		+
		U_0^\top\exp(-2\mathcal Ht)U_0
		\right]^{-1}
		U_0^\top\exp(-\mathcal Ht)
		\\
		&\quad =
		\left\{
		\exp(-\mathcal Ht)U_0
		\left[
		I_N-U_0^\top U_0
		+
		U_0^\top\exp(-2\mathcal Ht)U_0
		\right]^{-1/2}
		\right\}
		\\
		&\qquad\cdot
		\left\{
		\exp(-\mathcal Ht)U_0
		\left[
		I_N-U_0^\top U_0
		+
		U_0^\top\exp(-2\mathcal Ht)U_0
		\right]^{-1/2}
		\right\}^{\top}.
	\end{aligned}
	$$
Hence there exists \(Q(t)\in\mathcal O^N\) such that
	$$
	U(t)
	=
	\exp(-\mathcal Ht)U_0
	\left[
	I_N-U_0^\top U_0
	+
	U_0^\top\exp(-2\mathcal Ht)U_0
	\right]^{-1/2}
	Q(t).
	$$
	This proves \eqref{expression of solution(constant)}.
\end{proof}

\section{Asymptotic behavior}
This section studies the asymptotic behavior of the model as $t\rightarrow+\infty$, including asymptotic orthogonality, convergence to equilibrium, and convergence to the minimizer.
\subsection{Asymptotic orthogonality}
Theorem \ref{expression of solution(linear operator)} gives 
\begin{equation}
	U_0^\top U_0 = I_N \   \Rightarrow \  U(t)^\top U(t) = I_N, \quad \forall t\geqslant0.
\end{equation}
In fact, this indicates a tendency toward orthogonality even if the initial value is not orthogonal; that is, the solution is asymptotically orthogonal.

Let $\tilde{\mathcal{L}}_{0}$ denote the set of admissible initial conditions
$$
\tilde{\mathcal{L}}_{0}=
\left\{U\in \left(\mathcal{D}(\mathcal{H})\right)^N: U^\top \mathcal{H}U < 0\right\}.
$$
We first describe the asymptotic behavior of the singular values of the solution $U(t)$ to equation (\ref{operator of gradient flow based model}). The proof is provided in Appendix~\ref{proof of lemma: singular values of solution}.

\begin{lemma}
	\label{singular values of solution}
	Let $U(t)$ be the solution of (\ref{operator of gradient flow based model}). 
	If $U_0\in \tilde{\mathcal{L}}_{0}$, then there exist $N$ nonzero eigenvalues $\sigma_1(t) \leqslant \sigma_2(t) \leqslant \cdots \leqslant \sigma_N(t)$ of $U(t)^\top U(t)$ and there holds
	\begin{itemize}
		\item[(1)] $\lim \limits_{t \rightarrow \infty} \sigma_i(t)=1, \quad 1 \leqslant i \leqslant N$;
		\item[(2)] If $\sigma_i(0) \leqslant 1$ for $1 \leqslant i \leqslant N$, then $\sigma_i(t)$ is a nondecreasing function of $t$ and $ \sigma_i(t)\leqslant 1$ for $1 \leqslant i \leqslant N$;
		\item[(3)] If $\sigma_i(0) \geqslant 1$ for $1 \leqslant i \leqslant N$, then $\sigma_i(t)$ is a nonincreasing function of $t$ and $ \sigma_i(t)\geqslant 1$ for $1 \leqslant i \leqslant N$;
		\item[(4)] Let $\beta = \min\{1, \sigma_1(0)\}$. Then
		$$  \sigma_i(t)\geqslant \beta, \quad \forall i=1, 2,\cdots,N.$$ 
	\end{itemize}
	
\end{lemma}

\begin{lemma}
	\label{Bound of UHU}
	Let $U(t)$ be the solution of (\ref{operator of gradient flow based model}). If $U_0\in \tilde{\mathcal{L}}_{0}$, then
	$$
	U(t) \in \tilde{\mathcal{L}}_{0},\quad \forall t\geqslant0.
	$$
	Furthermore, there exists a constant $c_0>0$, depending only on $U_0$, such that
	$$
	U(t)^\top \mathcal{H}U(t)\leqslant -c_0 I_N,\quad \forall t\geqslant0.
	$$
\end{lemma}
\begin{proof}
	The proof is given in Appendix~\ref{proof of lemma: Bound of UHU}.
\end{proof}

Lemma \ref{Bound of UHU} shows that the solution $U(t)$ belongs to the set
\begin{equation*}
	\begin{aligned}
		\left\{U\in \big(\mathcal{D}(\mathcal{H})\big)^N : U^\top \mathcal{H}U \leqslant -c_0I_N\right\},
	\end{aligned}
\end{equation*}
where $c_0$ depends only on $U_0$. We then define the new admissible set
\begin{equation*}
	\begin{aligned}
		\mathcal{L}_{c_0} = \left\{[U]\subset\big(\mathcal{D}(\mathcal{H})\big)^N: U^\top \mathcal{H}U \leqslant -c_0I_N\right\},
	\end{aligned}
\end{equation*}
and we conclude that $[U(t)]$ remains in $	\mathcal{L}_{c_0}$ for all $t\geqslant0$.

Based on the above lemmas, we can establish the following theorem.
\begin{theorem}
	\label{limiting solution in stiefel manifold} 
	Let  $U(t)$ be the solution of (\ref{operator of gradient flow based model}). If $U_0\in \tilde{\mathcal{L}}_{0}$, then 
	\begin{equation}
		\label{Quasi-orthogonality}
		\begin{aligned}
			\left\|I_N-U(t)^\top   U(t)\right\| \leqslant\left\|I_N-U_0^\top U_0\right\| \exp \left(-2c_0 t\right), \quad 
			\forall t \geqslant 0.
		\end{aligned}
	\end{equation}
	In particular,
	$$
	\lim _{t \rightarrow \infty} U(t) ^\top U(t)=I_N.  
	$$
	
\end{theorem}

\begin{proof}
	It follows from Lemma \ref{Bound of UHU} that
	\begin{equation*}
		\begin{aligned}
			\lambda_{\max}\left(U(t)^\top \mathcal{H}U(t) \right)\leqslant -c_0.
		\end{aligned}
	\end{equation*}
	Then
	$$
	\begin{aligned}
		\frac{\mathrm{d}    }{\mathrm{d}  t} \left\|I_N-U(t)^\top  U(t)\right\|^2=& 4 \operatorname{tr}\left[\left(I_N-U(t)^\top  U(t)\right) \cdot {U(t)^\top  \mathcal{H} U(t)}\cdot\left(I_N-U(t)^\top  U(t)\right)\right]\\
		\leqslant& -4c_0\operatorname{tr}\left[\left(I_N-U(t)^\top  U(t)\right)^2 \right] \\
		\leqslant& -4c_0\left\|I_N-U(t)^\top  U(t)\right\|^2,
	\end{aligned}
	$$
	we derive from Gr\"{o}nwall's inequality that
	\begin{equation*}
		\left\|I_N-U(t)^\top  U(t)\right\|^2\leqslant \left\|I_N-U_0^\top U_0\right\|^2  \exp(-4c_0t), \quad \forall t \geqslant 0.
	\end{equation*}
This completes the proof.
\end{proof}
\begin{remark}
	The above theorem implies that the solution of \eqref{operator of gradient flow based model} tends to become orthogonal as $t$ grows.
	Moreover, once the solution $U(t)$ becomes orthogonal at some time $t_0\geqslant0$, it remains orthogonal for all subsequent times, namely,
	$$
	U(t)^\top U(t)=I_N,\qquad t\geqslant t_0.
	$$
	This result shows a structural property of the quasi-Grassmannian model. The initial data are not required to satisfy the Stiefel constraint exactly; instead, the continuous flow drives the solution toward the Stiefel manifold automatically, without an external projection. This property indicates the robustness of the continuous model with respect to deviations from orthogonality.
	
\end{remark}

\subsection{Gradient behavior}
In this subsection, we prove that the quasi-Grassmannian gradient flow model (\ref{operator of gradient flow based model}) will converge to the equilibrium state.

Let $U(t)$ be the solution of \eqref{refined weak formulation} and define the bounded self-adjoint operator
\[
\mathcal Z(t):(L^2(\Omega))^N\to(L^2(\Omega))^N
\]
by
\[
\mathcal Z(t)V=U(t)(U(t)^\top V),
\qquad
\forall V\in(L^2(\Omega))^N .
\]
Since \(U(t)\in(\mathcal D(\mathcal H))^N\), $\operatorname{Ran}\mathcal Z(t)\subset(\mathcal D(\mathcal H))^N$.
\begin{lemma}
	\label{convergence of Z}
	Let \(U(t)\) be the solution of \eqref{operator of gradient flow based model}.
	Assume that \(U_0\in\tilde{\mathcal L}_0\). Then there exist a self-adjoint operator
	\(\mathcal Z_\infty: \big(L^2(\Omega)\big)^N \to \big(L^2(\Omega)\big)^N\) and constants \(K>0\) and \(\gamma>0\) such that
	\begin{equation*}
		\begin{aligned}
			\|(\mathcal Z(t)-\mathcal Z_\infty)V\|
			+
			\|\mathcal H(\mathcal Z(t)-\mathcal Z_\infty)V\|
			\leqslant
			Ke^{-\gamma t}\|V\|,
			\qquad
			\forall V\in (L^2(\Omega))^N .
		\end{aligned}
	\end{equation*}
	Moreover,
	$$
	\mathcal Z_\infty^2=\mathcal Z_\infty ,
	$$
	and the range of \(\mathcal Z_\infty\) is invariant under \(\mathcal H\). Consequently,
	$$
	\mathcal H\mathcal Z_\infty
	=
	\mathcal Z_\infty\mathcal H
	=
	\mathcal Z_\infty\mathcal H\mathcal Z_\infty.
	$$
\end{lemma}
\begin{proof}
	The proof is deferred to Appendix~\ref{proof of lemma: convergence of Z}.
\end{proof}

The following theorem states that the gradient vanishes as $t\rightarrow \infty$ and exponentially decreases over time, which implies that the solution of the quasi-Grassmannian gradient flow model (\ref{operator of gradient flow based model}) exponentially converges to the equilibrium state.
\begin{theorem}
	\label{nabla_G E(U(t)) = 0}
	Let \(U(t)\) be the solution of \eqref{operator of gradient flow based model}. Assume that
	\(U_0\in\tilde{\mathcal L}_0\). Then there exist constants \(K>0\) and
	\(\gamma>0\) such that
	$$
	\|\nabla_GE(U(t))\|
	\leqslant
	Ke^{-\gamma t},
	\qquad t\geqslant0 .
	$$
	Consequently,
	$$
	\lim_{t\to\infty}\|\nabla_GE(U(t))\|=0 .
	$$
\end{theorem}

\begin{proof}
	By Lemma~\ref{convergence of Z}, there exist a self-adjoint projection $\mathcal Z_\infty$ and constants $K>0$ and $\gamma>0$ such that
	\begin{equation*}
		\begin{aligned}
		\|(\mathcal Z(t)-\mathcal Z_\infty)V\|
			+
			\|\mathcal H(\mathcal Z(t)-\mathcal Z_\infty)V\|
			\leqslant
			Ke^{-\gamma t}\|V\|,
			\qquad
			\forall t\geqslant0,\quad
			\forall V\in(L^2(\Omega))^N .
		\end{aligned}
	\end{equation*}
	Moreover,
	$$
	\mathcal H\mathcal Z_\infty V
	=
	\mathcal Z_\infty\mathcal H V,
	\qquad
	\forall V\in(\mathcal D(\mathcal H))^N,
	$$
	and
	$$
	\mathcal H\mathcal Z_\infty V
	=
	\mathcal Z_\infty\mathcal H\mathcal Z_\infty V,
	\qquad
	\forall V\in(L^2(\Omega))^N.
	$$

	By Lemma~\ref{singular values of solution}, there exists $\beta>0$ such that
	$$
	U(t)^\top U(t)\geqslant \beta I_N,
	\qquad t\geqslant0 .
	$$
	Thus $U(t)^\top U(t)$ is invertible and
	$$
	\left\|
	U(t)\big(U(t)^\top U(t)\big)^{-1}
	\right\|^2
	=
	\operatorname{tr}\left(\big(U(t)^\top U(t)\big)^{-1}\right)
	\leqslant
	\frac{N}{\beta}.
	$$
	Since $	\mathcal Z(t)U(t)=U(t)\big(U(t)^\top U(t)\big)$, we have
	\begin{equation}\label{eq: Z(t) formulation}
			U(t)=\mathcal Z(t)U(t)\big(U(t)^\top U(t)\big)^{-1}.
	\end{equation}
	Hence
	\begin{equation*}
		\begin{aligned}
			\nabla_GE(U(t))
			&=
			\mathcal HU(t)-\mathcal Z(t)\mathcal HU(t)
			\\
			&=
			\left[
			\mathcal H\mathcal Z(t)
			-
			\mathcal Z(t)\mathcal H\mathcal Z(t)
			\right]
			U(t)\big(U(t)^\top U(t)\big)^{-1}.
		\end{aligned}
	\end{equation*}
	Using
	$$
	\mathcal H\mathcal Z_\infty V
	=
	\mathcal Z_\infty\mathcal H\mathcal Z_\infty V,
	\qquad
	\forall V\in(L^2(\Omega))^N,
	$$
	we obtain
	\begin{equation*}
		\begin{aligned}
			\nabla_GE(U(t))
			={}&
			\mathcal H(\mathcal Z(t)-\mathcal Z_\infty)
			U(t)\big(U(t)^\top U(t)\big)^{-1}
			\\
			&-
			(\mathcal Z(t)-\mathcal Z_\infty)
			\mathcal H\mathcal Z(t)
			U(t)\big(U(t)^\top U(t)\big)^{-1}
			\\
			&-
			\mathcal Z_\infty\mathcal H(\mathcal Z(t)-\mathcal Z_\infty)
			U(t)\big(U(t)^\top U(t)\big)^{-1}.
		\end{aligned}
	\end{equation*}
	The first term is bounded by $Ke^{-\gamma t}$. Since $\mathcal Z_\infty$ is a projection and hence bounded, the third term is also bounded by $Ke^{-\gamma t}$.
	
	We obtain from \eqref{eq: Z(t) formulation} and Lemma~\ref{convergence of Z} that
	\begin{equation*}
		\begin{aligned}
			\|\mathcal HU(t)\|
			&=
			\left\|
			\mathcal H\mathcal Z(t)
			U(t)\big(U(t)^\top U(t)\big)^{-1}
			\right\|
			\\
			&\leqslant
			\left\|
			\mathcal H(\mathcal Z(t)-\mathcal Z_\infty)
			U(t)\big(U(t)^\top U(t)\big)^{-1}
			\right\|
			+
			\left\|
			\mathcal H\mathcal Z_\infty
			U(t)\big(U(t)^\top U(t)\big)^{-1}
			\right\|
			\\
			&\leqslant C .
		\end{aligned}
	\end{equation*}
	Here $\mathcal H\mathcal Z_\infty$ is bounded on $(L^2(\Omega))^N$ since $\mathcal Z_\infty$ has finite-dimensional range contained in eigenspaces of $\mathcal H$. Taking $V=\mathcal HU(t)$ in Lemma~\ref{convergence of Z} gives
	$$
	\left\|
	(\mathcal Z(t)-\mathcal Z_\infty)\mathcal HU(t)
	\right\|
	\leqslant
	Ke^{-\gamma t}\|\mathcal HU(t)\|
	\leqslant
	Ke^{-\gamma t}.
	$$
	Thus
	$$
	\|\nabla_GE(U(t))\|
	\leqslant
	Ke^{-\gamma t},
	\qquad t\geqslant0 .
	$$
	Consequently,
	$$
	\lim_{t\to\infty}\|\nabla_GE(U(t))\|=0 .
	$$
\end{proof}

\subsection{Energy behavior}
In this subsection, we prove that the energy associated with the solution of the model (\ref{operator of gradient flow based model}) decays exponentially over time. Moreover, we demonstrate that the solution of the model (\ref{operator of gradient flow based model}) converges to the minimizer of the energy in (\ref{minimization in Grassmann}), which corresponds to the solution of the eigenvalue problem.

For simplicity, in the subsequent analysis we define
\begin{equation*}
	\begin{aligned}
		\mathcal{M}^N_{\leqslant} = \left\{ U \in \big(H_0^1(\Omega)\big)^N : 0<U^\top U \leqslant I_N \right\}.
	\end{aligned}
\end{equation*}
Unlike $\mathcal{M}^N$, which demands the column vectors to be orthonormal, $\mathcal{M}^N_{\leqslant}$ only requires that the columns of $U$ are linearly independent and $ \lambda_i(U(t)^\top U(t))$ to be bounded within $ (0,1]$.
Correspondingly, we define the quotient set \( \mathcal{G}^N_{\leqslant} \) of \( \mathcal{M}^N_{\leqslant} \) as
$$\mathcal{G}^N_{\leqslant} = \mathcal{M}^N_{\leqslant}/ \sim ,$$
referred to as the quasi-Grassmannian.
Lemma \ref{singular values of solution} indicates that if the initial value $U_0\in \mathcal{M}^N_{\leqslant}$, then the solution $U(t)\in  \mathcal{M}^N_{\leqslant}$ for all $t\geqslant0$, that is, $[U(t)]\in  \mathcal{G}^N_{\leqslant}, \ \forall t\geqslant0$.

First, we show that the energy $E(U(t))$ decreases monotonically over time.
\begin{theorem}
	\label{energy decreases}
	Let $U(t)$ be the solution of (\ref{operator of gradient flow based model}). If $U_0\in \mathcal{M}_{\leqslant}^N \bigcap \tilde{\mathcal{L}}_{0}$, then
	\begin{equation*}
		\begin{aligned}
			\frac{\mathrm{d} E(U(t))}{\mathrm{d} t}\leqslant 0, \quad \forall t\geqslant 0,
		\end{aligned}
	\end{equation*}
	that is, the solution of (\ref{operator of gradient flow based model}) satisfies energy dissipation.
\end{theorem}
\begin{proof}
	By direct computation,
	\begin{equation*}
		\begin{aligned}
			&\|\nabla_GE(U)\|^2 - \operatorname{tr}\left(\nabla E(U)^\top \nabla_GE(U)\right)
			\\=&-\operatorname{tr}\left(\nabla E(U)^\top U U^\top \nabla_GE(U)\right)
			\\=&- \operatorname{tr}\left[ (U^\top\mathcal{H} U)^2 (I_N - U^\top U) \right]
			-\operatorname{tr}\left(\nabla E(U)^\top U U^\top \left( \nabla E(U)U^\top U - U U^\top \nabla E(U) \right)\right).
		\end{aligned}
	\end{equation*}
	Since $U^\top \left( \nabla E(U)U^\top U - U U^\top \nabla E(U) \right)$ is skew-symmetric and $\nabla E(U)^\top U$ is symmetric, the last term equals $0$.
	
	We see from Lemma \ref{singular values of solution} that $U(t)^\top U(t) \leqslant I_N, \forall t\geqslant 0$ with $U_0^\top U_0\leqslant I_N$. Then we have
	\begin{equation*}
		\begin{aligned}
			\frac{\mathrm{d} E(U(t))}{\mathrm{d} t}
			=&\frac{\delta E}{\delta U} \cdot \frac{\mathrm{d} U}{\mathrm{~d} t}=-\operatorname{tr}\left(\nabla E(U(t))^{\top} \nabla_G E\left(U(t)\right)\right)
			\\ =& - \operatorname{tr}\left[ (U^\top \mathcal{H} U)^2 (I_N - U^\top U) \right] -\|\nabla_GE(U)\|^2
			\\\leqslant&-\|\nabla_GE(U)\|^2\leqslant 0.
		\end{aligned}
	\end{equation*}
	This establishes the result.
\end{proof}

In the following analysis, we assume that the initial value satisfies
$$[U_0] \in B([V^{(*)}], \delta_1)$$
for some $\delta_1\in (0,1)$.
We further explain this initial condition in the following lemma. The proof is given in Appendix~\ref{proof of lemma: meaning of initial condition}.
\begin{lemma}
	\label{meaning of initial condition}
	If $U \in B([V^{(*)}], \delta)$ for any given $\delta\in (0,1)$, then $U$ contains the component of all the eigenfunctions corresponding to the $N$ smallest eigenvalues of $\mathcal{H}$.
\end{lemma}

The minimizer $[V^{(*)}]\in \mathcal{G}^N$ is the unique critical point of (\ref{minimization in Grassmann}). Thus $[V^{(*)}]$ is the unique point in $B\left(\left[V^{(*)}\right], \delta_1\right)\bigcap \mathcal{G}^N$.
By the definitions of the quasi-Grassmannian $\mathcal{G}^N_{\leqslant}$ and the admissible set $\mathcal{L}_{c_0}$, this uniqueness extends to the region
$$B\left(\left[V^{(*)}\right], \delta_1\right)\bigcap \mathcal{G}^N_{\leqslant}\bigcap \mathcal{L}_{c_0}. $$
Thus $[V^{(*)}]$ is also the unique critical point in this set.

For a fixed constant $\delta_2\in \left(0, \delta_1\right) $, we define
\begin{equation*}
	E_0 = \inf\left\{ E(\tilde{U}): [\tilde{U}] \in \overline{ B\left(\left[V^{(*)}\right], \delta_1\right) \backslash B\left(\left[V^{(*)}\right], \delta_2\right)}\bigcap \mathcal{G}^N_{\leqslant}\bigcap\mathcal{L}_{c_0}   \right\}.
\end{equation*}
We present the convergence theorem as follows.
\begin{theorem}
	\label{convergence}
	Let $U(t)$ be the solution of (\ref{operator of gradient flow based model}). If $U_0 \in \mathcal{M}^N_{\leqslant}\bigcap\tilde{\mathcal{L}}_{0} $ and
	$$E(U_0) \leqslant \frac{E_0+E(V^{(*)})}{2} \stackrel{\Delta }{=} E_1,$$
	then
	\begin{equation*}
		\begin{aligned}
			&\lim\limits_{t\rightarrow \infty}E(U(t)) = E(V^{(*)}),
			\\& \lim\limits_{t\rightarrow \infty}\operatorname{dist}_1\left([U(t)], [V^{(*)}]\right) =0.
		\end{aligned}
	\end{equation*}
\end{theorem}
\begin{proof}
	Define
	$$
	\mathcal S=
	\left\{
		U\in\big(H_0^1(\Omega)\big)^N:
		[U]\in B\left(\left[V^{(*)}\right],\delta_2\right)
		\bigcap\mathcal G^N_{\leqslant}
		\bigcap\mathcal L_{c_0}
		\bigcap\mathfrak L_{E_1}
		\right\},
	$$
	where
	$$
	\mathfrak L_{E_1}
	=
	\left\{
		[U]\subset\big(H_0^1(\Omega)\big)^N:
		E(U)\leqslant E_1
		\right\}.
	$$
	Since $E(V^{(*)})<E_0$, we have $E_1<E_0$. By the assumption $[U_0]\in B([V^{(*)}],\delta_1)$ and the definition of $E_0$, the bound $E(U_0)\leqslant E_1$ implies
	$$
	[U_0]\in B([V^{(*)}],\delta_2).
	$$

	By Lemma~\ref{singular values of solution} and Lemma~\ref{Bound of UHU}, we obtain
	$$
	U(t)\in\mathcal M^N_{\leqslant},\qquad U(t)^\top\mathcal HU(t)\leqslant -c_0I_N,\qquad t\geqslant0.
	$$
	Combining this with Theorem~\ref{energy decreases}, we have
	$$
	E(U(t))\leqslant E(U_0)\leqslant E_1,\qquad t\geqslant0.
	$$
	Thus, by the definition of $E_0$,
	$$
	[U(t)]\in B([V^{(*)}],\delta_2),\qquad t\geqslant0.
	$$
	Consequently,
	$$
	U(t)\in\mathcal S,\qquad t\geqslant0.
	$$
	
	Set
	$$
	V(t)=U(t)\big(U(t)^\top U(t)\big)^{-1}.
	$$
	By Lemma~\ref{singular values of solution}, there exists $\beta>0$ such that
	$$
	U(t)^\top U(t)\geqslant \beta I_N,\qquad t\geqslant0.
	$$
	Hence
	$$
	\|V(t)\|^2=\operatorname{tr}\left((U(t)^\top U(t))^{-1}\right)\leqslant \frac{N}{\beta},\qquad t\geqslant0.
	$$
	Due to $U(t)=\mathcal Z(t)V(t)$, we obtain from Lemma~\ref{convergence of Z} that
	$$
	\|U(t)-\mathcal Z_\infty V(t)\|+\|\mathcal H(U(t)-\mathcal Z_\infty V(t))\|\leqslant Ce^{-\gamma t}.
	$$
	The range of $\mathcal Z_\infty$ is finite-dimensional and contained in $(\mathcal D(\mathcal H))^N$. Since $\{V(t)\}_{t\geqslant0}$ is bounded in $(L^2(\Omega))^N$, the set $\{\mathcal Z_\infty V(t):t\geqslant0\}$
	is relatively compact in $(\mathcal D(\mathcal H))^N$. Hence, for every sequence $\tau_k\to\infty$, there exist a subsequence, still denoted by $\tau_k$, and $\hat U\in(\mathcal D(\mathcal H))^N$ such that
	$$
	\mathcal Z_\infty V(\tau_k)\to\hat U
	\quad
	\text{in }(\mathcal D(\mathcal H))^N.
	$$
	Consequently,
	$$
	U(\tau_k)\to\hat U
	\quad
	\text{in }(\mathcal D(\mathcal H))^N,
	$$
	which together with Lemma \ref{lem:compact embedding} leads to
	$$
	U(\tau_k)\to\hat U
	\quad
	\text{in }\big(H_0^1(\Omega)\big)^N.
	$$
	Since $U(\tau_k)\in\mathcal S$ and $\mathcal S$ is closed in $\big(H_0^1(\Omega)\big)^N$, we have $\hat U\in\mathcal S$.
	
	The map
	$$
	U\mapsto \nabla_GE(U)=\mathcal HU-U(U^\top\mathcal HU)
	$$
	is continuous from bounded subsets of $(\mathcal D(\mathcal H))^N$, equipped with the graph norm, into $(L^2(\Omega))^N$. Therefore
	$$
	\nabla_GE(U(\tau_k))
	\to
	\nabla_GE(\hat U)
	\quad
	\text{in }(L^2(\Omega))^N.
	$$
	By Theorem~\ref{nabla_G E(U(t)) = 0},
	$$
	\|\nabla_GE(U(\tau_k))\|
	\leqslant
	Ke^{-\gamma\tau_k}
	\to0.
	$$
	Thus
	$$
	\nabla_GE(\hat U)=0.
	$$
	Note that $\hat U\in\mathcal S\subset B([V^{(*)}],\delta_1)\cap\mathcal G^N_{\leqslant}\cap\mathcal L_{c_0}$ and $[V^{(*)}]$ is the unique critical point in this set, we arrive at
	$$
	[\hat U]=[V^{(*)}].
	$$
	
	The energy $E(U(t))$ is nonincreasing and bounded from below by $E(V^{(*)})$. Hence $\lim_{t\to\infty}E(U(t))$ exists. Applying the preceding argument to any sequence $\tau_k\to\infty$ gives, after taking a subsequence,
	$$
	E(U(\tau_k))\to E(\hat U)=E(V^{(*)}).
	$$
	Therefore
	$$
	\lim_{t\to\infty}E(U(t))=E(V^{(*)}).
	$$
	
It remains to prove convergence in the $H_0^1$-Grassmannian distance. We argue by contradiction. Then there exist $\varepsilon>0$ and a sequence $\tau_p\to\infty$ such that
	$$
	\operatorname{dist}_1([U(\tau_p)],[V^{(*)}])\geqslant\varepsilon.
	$$
	By the compactness argument above, there exist a subsequence $\tau_{p_q}\to\infty$ and $\bar U\in(\mathcal D(\mathcal H))^N$ such that
	$$
	U(\tau_{p_q})\to\bar U
	\quad
	\text{in }(\mathcal D(\mathcal H))^N.
	$$
	In particular,
	$$
	\lim_{q\to\infty}\|U(\tau_{p_q})-\bar U\|_1=0.
	$$
	Since $U(\tau_{p_q})\in\mathcal S$ and $\mathcal S$ is closed in $\big(H_0^1(\Omega)\big)^N$, we have $\bar U\in\mathcal S$. As above,
	$$
	\nabla_GE(\bar U)=0.
	$$
	By uniqueness of the critical point in $\mathcal S$,
	$$
	[\bar U]=[V^{(*)}].
	$$
	Thus
	$$
	\operatorname{dist}_1([U(\tau_{p_q})],[V^{(*)}])\to0,
	$$
	which contradicts the choice of $\tau_p$. Hence
	$$
	\lim_{t\to\infty}
	\operatorname{dist}_1([U(t)],[V^{(*)}])
	=
	0.
	$$
\end{proof}

The theorem shows that the equivalence class \([U(t)]\) converges to \([V^{(*)}]\) in the \(H_0^1\)-Grassmannian distance and that the energy converges to the minimum value \(E(V^{(*)})\).
Furthermore, we can show that the solution of \eqref{refined weak formulation} converges exponentially to the minimizer as time progresses.
\begin{corollary}
	Under the same hypotheses as in Theorem \ref{convergence},
	there hold
	\begin{equation*}
		\begin{aligned}
			& \operatorname{dist}\left([U(t)], [V^{(*)}]\right) \leqslant \frac{K}{\gamma} \exp (-\gamma t),
			\\&	E(U(t)) - E(V^{(*)}) \leqslant \frac{|\lambda_1|^2}{2c_0} \left\| I_N-U_0^\top U_0 \right\|\exp(-2c_0t) +\frac{K^2}{2\gamma} \exp (-2\gamma t).
		\end{aligned}
	\end{equation*}
	where $\gamma>0$ and $K>0$ are constants mentioned in Theorem \ref{nabla_G E(U(t)) = 0}.
\end{corollary}
\begin{proof}
	It follows that for all $s_2>s_1 \geqslant 0$,
	$$
	\begin{aligned}
		\left\|U\left(s_2\right)-U\left(s_1\right)\right\|
		\leqslant &\int_{s_1}^{s_2} 	\|\nabla_G E(U)\| \mathrm{d} t
		\\ \leqslant &\int_{s_1}^{s_2} K \exp (-\gamma t) \mathrm{d} t 
		=\frac{K}{\gamma}\left[\exp \left(-\gamma s_1\right)-\exp \left(-\gamma s_2\right)\right].
	\end{aligned}
	$$
	
	If $t$ is sufficiently large and $s_2, s_1 \geqslant t$, then $\lim\limits_{t\rightarrow \infty}\left\|U\left(s_2\right)-U\left(s_1\right)\right\| =0$. It is clear to see $\lim\limits _{t \rightarrow \infty} U(t)$ exists in $\big(L^2(\Omega)\big)^N$ and is an equilibrium point of (\ref{operator of gradient flow based model}). 
	
	Letting $s_2$ tend to infinity and setting $s_1=t$, we obtain
	\begin{equation*}
		\left\|U(t)-\lim\limits _{t \rightarrow \infty} U(t)\right\| \leqslant \frac{K}{\gamma} \exp (-\gamma t).
	\end{equation*}
	Thus, 
	\begin{equation*}
		\begin{aligned}
			\operatorname{dist}\left([U(t)], [V^{(*)}]\right) =& \inf\limits_{Q\in \mathcal{O}^N} \left\|U(t)Q-V^{(*)}\right\|
			\\ \leqslant&\left\|U(t)-\lim\limits _{t \rightarrow \infty} U(t)\right\| \leqslant \frac{K}{\gamma} \exp (-\gamma t).
		\end{aligned}
	\end{equation*}
	
	Therefore,
	\begin{equation*}
		\begin{aligned}
			\operatorname{tr}\left[ (U^\top \mathcal{H} U)^2 (  I_N - U^\top U) \right]
			\leqslant & \|U^\top \mathcal{H} U\|^2 \|   I_N- U^\top U\|
			\\\leqslant& |\lambda_1|^2 \|   I_N- U^\top U\|,
		\end{aligned}
	\end{equation*}
	and
	\begin{equation*}
		\begin{aligned}
			\frac{\mathrm{d} E(U(t))}{\mathrm{d} t}
			=& - \operatorname{tr}\left[ (U^\top \mathcal{H} U)^2 (  I_N - U^\top U) \right] -\|\nabla_GE(U)\|^2
			\\\geqslant& -|\lambda_1|^2\|   I_N- U^\top U\|  -\|\nabla_GE(U)\|^2.
		\end{aligned}
	\end{equation*}
	This yields
	\begin{equation*}
		\begin{aligned}
			E(U(t)) - \lim\limits_{t\rightarrow\infty}E(U ) \leqslant \int_{t}^{\infty} |\lambda_1|^2 \|   I_N- U^\top U\|+ \|\nabla_GE(U(s))\|^2   \mathrm{d} s ,
		\end{aligned}
	\end{equation*}
	that is,
	\begin{equation*}
		\begin{aligned}
			E(U(t)) - E(V^{(*)} ) \leqslant \frac{|\lambda_1|^2}{2c_0} \left\| I_N-U_0^\top U_0 \right\|\exp(-2c_0t) +\frac{K^2}{2\gamma} \exp (-2\gamma t),
		\end{aligned}
	\end{equation*}
	which means the energy error decays exponentially as $t$ increases. This concludes the proof.
\end{proof}

Under the assumption 
\begin{equation}
	\label{expression of Hess}
	\begin{aligned}
		&	\left\langle \nabla^2 E(U) \nabla_G E(U) - \nabla_G E(U)  U^\top \nabla E(U)  , \nabla_G  E(U)\right\rangle 
		\\&\qquad\qquad \qquad\qquad\qquad \qquad \qquad \geqslant \mu \| \nabla_G  E(U) \|^2 ,
		\quad \forall [U] \in B\left(\left[V^{(*)}\right], \delta_2\right)
	\end{aligned}
\end{equation}
for some $\mu >0$, we can provide a more precise estimate of the exponential decay rate of $\|\nabla_GE(U)\|$.
\begin{remark}
By \cite{dai2017conjugate,edelman1998geometry}, the tangent space on $\mathcal{G}^N$ is
	$$
	\mathcal{T}_{[U]} \mathcal{G}^N=\left\{W \in\left(H_0^1\left(\Omega\right)\right)^N: W^{\top} U=0\right\},
	$$
	and the Hessian of $E(U)$ on $\mathcal{G}^N$ is
	$$
	\begin{aligned}
		\operatorname{Hess}_G E(U)[V, W]=\left\langle \nabla^2 E(U) W- W U^{\top} \nabla E(U), V \right\rangle, \quad \forall V, W \in \mathcal{T}_{[U]} \mathcal{G}^N. 
	\end{aligned}
	$$
	Moreover, the condition $\operatorname{Hess}_G E(U)[W, W] \geqslant \mu \|W\|^2, \quad \forall W \in \mathcal{T}_{[U]} \mathcal{G}^N $, has already been used in \cite{dai2017conjugate,dai2020,schneider2009direct}.
	
		It is clear that $\nabla_{G}E(U)^\top U  = 0, \forall U \in \mathcal{M}^N$; that is, $\nabla_{G}E(U)\in \mathcal{T}_{U} \mathcal{G}^N $.
	Thus (\ref{expression of Hess}) can be viewed as an extension of the Hessian to the space $\big(H_0^1(\Omega)\big)^N$, denoted by $\operatorname{Hess}_{\tilde{G}} E(U)\left[\cdot, \cdot \right]$. Indeed, (\ref{expression of Hess}) is equivalent to
	\begin{equation*}
		\begin{aligned}
			&\operatorname{Hess}_{\tilde{G}} E(U)\left[\nabla_{G}  E(U) , \nabla_{G}  E(U) \right]
			\\	=&\left\langle \nabla^2 E(U) \nabla_{G} E(U) 
			- \nabla_{G} E(U)     U^\top \nabla E(U),\nabla_{G}  E(U)   \right\rangle
			\\ \geqslant &\mu \| \nabla_{G}  E(U) \|^2.
		\end{aligned}
	\end{equation*}
\end{remark}
\begin{corollary}	
	\label{convergence rate}
	Let $U(t)$ be the solution of (\ref{operator of gradient flow based model}). Suppose $U_0 \in \tilde{\mathcal{L}}_{0} $ and $U_0^\top U_0 \leqslant I_N$. If the assumption (\ref{expression of Hess}) holds, then
	\begin{equation*}
		\begin{aligned}
			&	\left\|\nabla_GE(U(t))\right\|^2
			\\ \leqslant& \left\{
			\begin{aligned}
				&	\left\|\nabla_GE(U_0)\right\|^2 e^{-2\mu t} + 2 C^* \|I_N -U_0^\top U_0\| t e^{-2\mu t} \qquad\qquad\qquad\qquad \  \mu =c_0,
				\\	& 	\left\|\nabla_GE(U_0)\right\|^2 e^{-2\mu t} + \frac{C^* \|I_N -U_0^\top U_0\|}{\mu -c_0} \left( e^{\left(2\mu-2c_0\right)t}  -1\right) e^{-2\mu t} \qquad \mu\neq c_0.
			\end{aligned}\right.
		\end{aligned}
	\end{equation*}
\end{corollary}
\begin{proof}
	The proof is deferred to Appendix~\ref{proof of corollary: convergence rate}.
\end{proof}

Building on Corollary \ref{convergence rate}, we can derive a more precise estimate for the exponential decay rate of \(\operatorname{dist}([U(t)], [V^{(*)}])\) and $E(U(t))$ discussed above. To avoid redundancy, we omit unnecessary details here.

\section{Concluding remarks}
In this paper, we have proposed and analyzed a quasi-Grassmannian gradient flow model for eigenvalue problems of linear operators in an infinite-dimensional setting. We have established the well-posedness of the model and derived an analytic representation of its solution. We have also proved that the solution becomes asymptotically orthogonal and approaches the Stiefel manifold even when the initial data are non-orthogonal.

For the numerical solution of the quasi-Grassmannian gradient flow model, we propose and analyze in our related work \cite{wang2026quasi} a quasi-orthogonal approximation algorithm that preserves the structure of the model. We also plan to extend the framework to nonlinear eigenvalue problems and to validate the approach through numerical experiments. These efforts aim to further enhance the practical applicability of the model and broaden its use in operator-related eigenvalue problems.

\section*{Acknowledgments}
The authors thank Professor Xiaoying Dai and Professor Bin Yang for insightful discussions, and Professor Tianyang Chu and Dr. Yan Li for their careful review of the manuscript. The authors also thank the referee for the constructive comments and helpful suggestions, which have improved the presentation of this paper.

\appendix
\section{Detailed proofs}\label{app:proofs}
\subsection{Proof of Lemma~\ref{lemma: UHU<0}}\label{proof of lemma: UHU<0}
\begin{proof}
	Let
	$$
	A_m(t)=\left(U^{(m)}(t)\right)^\top\mathcal H U^{(m)}(t).
	$$
	Since \(\mathcal H\) is self-adjoint on \((\mathcal V_m)^N\), \(A_m(t)\) is symmetric. 
	
	Since \(U^{(m)}(t)\in(\mathcal V_m)^N\) and \(\mathcal V_m\) is spanned by eigenfunctions of \(\mathcal H\),
	$$
	\mathcal H U^{(m)}(t)\in(\mathcal V_m)^N.
	$$
	Moreover,
	$$
	U^{(m)}(t)\left(U^{(m)}(t)^\top\mathcal H U^{(m)}(t)\right)\in(\mathcal V_m)^N.
	$$
	Hence
	$$
	\mathcal P_m\nabla_GE(U^{(m)}(t))
	=
	\nabla_GE(U^{(m)}(t))
	=
	\mathcal H U^{(m)}(t)-U^{(m)}(t)A_m(t).
	$$
	Therefore,
	\begin{equation}
		\label{UHU<0}
		\begin{aligned}
			\frac{\mathrm{d}}{\mathrm{d}t}A_m(t)
			&=
			-\left(\mathcal H U^{(m)}-U^{(m)}A_m\right)^\top\mathcal H U^{(m)}
			-
			\left(U^{(m)}\right)^\top\mathcal H
			\left(\mathcal H U^{(m)}-U^{(m)}A_m\right)
			\\
			&=
			-2\left(\mathcal H U^{(m)}\right)^\top\mathcal H U^{(m)}
			+
			2A_m(t)^2.
		\end{aligned}
	\end{equation}
	Since \(\left(\mathcal H U^{(m)}\right)^\top\mathcal H U^{(m)}\) is positive semidefinite, \eqref{UHU<0} gives the differential inequality
	$$
	\frac{\mathrm{d}}{\mathrm{d}t}A_m(t)-2A_m(t)^2\leqslant0.
	$$
	
	Let \(\Phi(t)\in\mathbb R^{N\times N}\) be the solution of
	$$
	\frac{\mathrm{d}}{\mathrm{d}t}\Phi(t)=-A_m(t)\Phi(t),
	\qquad
	\Phi(0)=I_N.
	$$
 Since \(A_m(t)\) is symmetric,
	\begin{equation*}
		\begin{aligned}
			\frac{\mathrm{d}}{\mathrm{d}t}
			\left(\Phi(t)^\top A_m(t)\Phi(t)\right)
			&=
			\Phi(t)^\top
			\left(
			\frac{\mathrm{d}}{\mathrm{d}t}A_m(t)-2A_m(t)^2
			\right)
			\Phi(t)
			\leqslant0.
		\end{aligned}
	\end{equation*}
	Integrating from \(0\) to \(t\) yields
	$$
	\Phi(t)^\top A_m(t)\Phi(t)
	\leqslant
	A_m(0)
	=
	\left(\mathcal P_mU_0\right)^\top\mathcal H\mathcal P_mU_0
	\leqslant0.
	$$
	For any \(x\in\mathbb R^N\), set \(y=\Phi(t)^{-1}x\). Hence
	$$
	x^\top A_m(t)x
	=
	y^\top\Phi(t)^\top A_m(t)\Phi(t)y
	\leqslant0.
	$$
	Thus \(A_m(t)\leqslant0\), namely,
	$$
	\left(U^{(m)}(t)\right)^\top\mathcal H U^{(m)}(t)\leqslant0
	$$
	for all \(t\geqslant0\).
\end{proof}

\subsection{Proof of Proposition~\ref{existence of weak formulation of Galerkin equation}}\label{proof of proposition:existence of weak formulation of Galerkin equation}
\begin{proof}
	The weak formulation of (\ref{Galerkin equation}) is as follows
	\begin{equation}
		\label{weak formulation of Galerkin equation}
		\left\{\begin{aligned}
		&  \frac{\mathrm{d}    }{\mathrm{d}  t}\left(U^{(m)}(t), V\right)+\left\langle \mathcal{P}_m\nabla_GE(U^{(m)}(t)), V\right\rangle = 0, \quad \forall V\in \left(\mathcal{V}_m\right)^N,
			\\&U^{(m)}(0) = \mathcal{P}_mU_0,
		\end{aligned}\right.
	\end{equation}
	which is equivalent to 
	\begin{equation}
		\label{weak formulation of Galerkin equation in terms}
		\begin{aligned}
			\frac{\mathrm{d}    }{\mathrm{d}  t} \left(U^{(m)}(t), V_{jq}\right)+\left\langle \mathcal{P}_m\nabla_GE(U^{(m)}(t)), V_{jq}\right\rangle = 0, \  \forall j = 1,2,\cdots,m; q = 1, 2,\cdots,N,
		\end{aligned}
	\end{equation}
	where $U^{(m)}(t) = \sum\limits_{i=1}^{m} \sum\limits_{p=1}^N \alpha_{ip}^{(m)}(t) V_{ip} $. By the boundedness of $\mathcal{H}$ in the finite dimensional space $\mathcal{V}_m$, $\left\langle \mathcal{P}_m\nabla_GE(U), V_{jq}\right\rangle $ is continuous and satisfies a local Lipschitz condition with respect to $U\in \mathcal{V}_m$. 
	By the Picard--Lindel\"{o}f theorem (see Chapter~2, Theorem~2.2 of \cite{teschl2012ordinary}), there exists a unique solution $U^{(m)}(t)$ defined on some time interval $[0, T_m^*)$, where $T_m^*>0$ depends on the dimension $m$.

	It follows from $\left(U^{(m)}(0)\right)^\top \mathcal{H}U^{(m)}(0)\leqslant0$ that $\left(U^{(m)}(t)\right)^\top \mathcal{H}U^{(m)}(t)\leqslant0$. Therefore,
	\begin{equation*}
		\begin{aligned}
			&	\left\langle  \nabla_G  E\left( U^{(m)}(t)\right), U^{(m)}(t)\right\rangle  
			\\ =& 	\operatorname{tr}\left(\left(U^{(m)}(t)\right)^\top \mathcal{H} U^{(m)}(t)\right) -\operatorname{tr}\left(\left(U^{(m)}(t)\right)^\top U^{(m)}(t)\left(U^{(m)}(t)\right)^\top \mathcal{H}U^{(m)}(t)\right)
			\\ \geqslant &\left\langle\mathcal{H} U^{(m)}(t), U^{(m)}(t)\right\rangle
			\geqslant c_1\|\nabla U^{(m)}(t)\|^2 -c_2\|U^{(m)}(t)\|^2.
		\end{aligned}
	\end{equation*}
	
	We see from (\ref{weak formulation of Galerkin equation}) that
	\begin{equation*}
		\begin{aligned}
			\frac{1}{2} \frac{\mathrm{d}    }{\mathrm{d}  t}\left\|U^{(m)}(t)\right\|^2 =&\left\langle \frac{\mathrm{d}    }{\mathrm{d}  t}U^{(m)}(t), U^{(m)}(t) \right\rangle
			= - \left\langle \mathcal{P}_m \nabla_G E( U^{(m)}(t)),U^{(m)}(t)\right\rangle 
			\\=& \scalebox{0.99}{$ -  \left\langle   \nabla_G E(U^{(m)}(t)), U^{(m)}(t) \right\rangle
			\leqslant   -c_1  \left\|\nabla U^{(m)}(t)\right\|^2 +c_2\left\|U^{(m)}(t)\right\|^2, 	$}
		\end{aligned}
	\end{equation*}
	that is,
	\begin{equation}
		\label{Assumption derived inequalities}
		\begin{aligned}
			\frac{1}{2} \frac{\mathrm{d}    }{\mathrm{d}  t}\left\|U^{(m)}(t)\right\|^2 +c_1 \left\|                U^{(m)}(t)\right\|_1^2  \leqslant (c_1+c_2)\left\|U^{(m)}(t)\right\|^2.
		\end{aligned}
	\end{equation}
	Thus
	\begin{equation*}
		\begin{aligned}
			\frac{1}{2} \frac{\mathrm{d}    }{\mathrm{d}  t}\left\|U^{(m)}(t)\right\|^2   \leqslant (c_1+c_2)\left\|U^{(m)}(t)\right\|^2
		\end{aligned}
	\end{equation*} 
	and $ \left\|U^{(m)}(0)\right\|=  \left\|\mathcal{P}_mU_0\right\| \leqslant  \left\| U_0\right\|$. It follows from Gr\"{o}nwall's inequality that
	\begin{equation}\label{eq: uniform bound of Um}
		\begin{aligned}
			\left\|U^{(m)}(t)\right\|^2\leqslant \exp\left(2\left(c_1+c_2\right)t\right) \left\| U_0\right\|^2, \quad \forall t\in [0,T_m^*),
		\end{aligned}
	\end{equation}
	which implies that $\left\|U^{(m)}(t)\right\|^2$ is bounded for all $t\in [0,T_m^*)$.
	
	Integrating (\ref{Assumption derived inequalities}) yields
	\begin{equation*}
		\scalebox{0.98}{$	\frac{1}{2}\left\|U^{(m)}(T_m^*)\right\|^2 +c_1 \int_{0}^{T_m^*}  \left\|U^{(m)}(t)\right\|_1^2 \mathrm{d} t
			 \leqslant 	\frac{1}{2}\left\|U^{(m)}(0)\right\|^2 +\left(c_1+c_2\right) \int_{0}^{T_m^*}  \left\|U^{(m)}(t)\right\|^2 \mathrm{d} t ,	$}
	\end{equation*}
	which implies $\left\|U^{(m)}\right\|_{L^{2}\left(0, T_m^* ; \big(H_0^1(\Omega)\big)^N\right) }^2 = \int_{0}^{T_m^*}  \left\|U^{(m)}(t)\right\|_1^2 \mathrm{d} t<+\infty$, that is, $U^{(m)} \in L^{2}\left(0, T_m^* ; \big(H_0^1(\Omega)\big)^N\right)$.
	
	Since
	\begin{equation*}
		\begin{aligned}
			\lambda_1 	\left(U^{(m)}\right)^\top U^{(m)} 	\leqslant	\left(U^{(m)}\right)^\top \mathcal{H}U^{(m)} \leqslant 0 ,
		\end{aligned}
	\end{equation*}
	we have
	\begin{equation*}
		\begin{aligned}
			\left\| \left(U^{(m)}\right)^\top \mathcal{H}U^{(m)}\right\|\leqslant	|\lambda_1| 	\left\| \left(U^{(m)}\right)^\top U^{(m)}\right\|	\leqslant|\lambda_1| \left\|U^{(m)}\right\|^2.
		\end{aligned}
	\end{equation*}
	
	For any $V \in \big(H_0^1(\Omega)\big)^N$, it follows that
	\begin{equation*}
		\begin{aligned}
			\left|\left\langle  \mathcal{P}_m\nabla_G E\left( U^{(m)}\right), V\right\rangle \right| 
			=& \left|\left\langle  \nabla_G E\left( U^{(m)}\right), \mathcal{P}_mV\right\rangle \right| 
			\\ \leqslant &\scalebox{0.99}{$ \left|\left\langle  \mathcal{H} U^{(m)},   \mathcal{P}_mV \right\rangle\right| + \left|\operatorname{tr}\left(  \left( \mathcal{P}_mV\right)^\top U^{(m)} \left(U^{(m)}\right)^\top \mathcal{H}U^{(m)} \right) \right| $}
			\\ \leqslant &\scalebox{0.97}{$C_b\left\|\mathcal{P}_mV\right\|_1 \left\|U^{(m)}\right\|_1 +  \left\| \left( \mathcal{P}_mV\right)^\top U^{(m)} \right\| \left\| \left(U^{(m)}\right)^\top \mathcal{H}U^{(m)}\right\|.$}
		\end{aligned}
	\end{equation*}
	Thus, 
	\begin{equation*}
		\begin{aligned}
			\left|\left\langle  \mathcal{P}_m\nabla_G E\left( U^{(m)}\right), V\right\rangle \right| 
			\leqslant & C_b\left\|\mathcal{P}_mV\right\|_1 \left\|U^{(m)}\right\|_1 + |\lambda_1|\left\|\mathcal{P}_mV\right\| \left\|U^{(m)}\right\| \left\|U^{(m)}\right\|^2
			\\ \leqslant &\left(C_b+  |\lambda_1| \left\|U^{(m)}\right\|^2 \right)\cdot \left\|\mathcal{P}_mV\right\|_1 \left\|U^{(m)}\right\|_1.
		\end{aligned}
	\end{equation*}
	By the boundedness of $\left\|U^{(m)}(t)\right\|$ and $\mathcal{P}_m$, there exists a constant $\tilde{C}_b>0$ such that
	\begin{equation*}
		\begin{aligned}
			\frac{\left|\left\langle  \mathcal{P}_m\nabla_G E\left( U^{(m)}(t)\right), V\right\rangle \right|}{\|V\|_1}  \leqslant \tilde{C}_b\left\|U^{(m)}(t)\right\|_1 , \quad \forall t\in [0, T_m^*).
		\end{aligned}
	\end{equation*}
	Consequently,
	\begin{equation*}
		\begin{aligned}
			\left\|\frac{\mathrm{d}    }{\mathrm{d}  t} U^{(m)}\right\|_{L^{2}\left(0, T_m^* ; \left(H^{-1}(\Omega)\right)^N\right) }^2 & = \int_{0}^{T_m^*}	\left\|\frac{\mathrm{d}    }{\mathrm{d}  t} U^{(m)}(t)\right\|_{-1}^2 \mathrm{d}t  
			\\&\scalebox{0.99}{$=\int_{0}^{T_m^*}\left( \sup\limits_{V \in \left(H_0^1(\Omega)\right)^N \backslash\{0\}}\frac{	\left| \left\langle  \mathcal{P}_m  \nabla_G E\left( U^{(m)}(t)\right), V\right\rangle  \right|}{\|V\|_1}  \right)^2\mathrm{d} t $}
			\\ &\leqslant  \scalebox{0.99}{$	\tilde{C}_b^2\int_{0}^{T_m^*}	\left\|U^{(m)}(t)\right\|_1^2 \mathrm{d}t   =\tilde{C}_b^2 \left\|U^{(m)}\right\|_{L^{2}\left(0, T_m^* ; \left(H_0^1(\Omega)\right)^N\right) }^2 , $}
		\end{aligned}
	\end{equation*}
	which implies that $	\left\|\frac{\mathrm{d}    }{\mathrm{d}  t} U^{(m)}\right\|_{L^{2}\left(0, T_m^* ; \big(H^{-1}(\Omega)\big)^N\right) }$ is finite. Thus, we come to the conclusion that $\frac{\mathrm{d}   }{\mathrm{d}  t}U^{(m)} \in L^{2}\left(0, T_m^* ; \big(H^{-1}(\Omega)\big)^N\right)$.
\end{proof}

\subsection{Proof of Corollary~\ref{T be infty}}
\label{proof of corollary: T be infty}
\begin{proof}
	To prove the conclusion by contradiction, suppose $T_m^*< T$ is finite and maximal in the sense that the model \eqref{Galerkin equation} is no longer well-posed for $t\geqslant T_m^*$.
	
	By H\"{o}lder's inequality, for any $t_1,t_2\in [0,T_m^*)$, we derive the following estimate:
	\begin{equation*}
		\begin{aligned}
			\left\| U^{(m)}(t_2) -U^{(m)} (t_1)\right\|_{-1}^2 
			&\leqslant \left( \int_{t_1}^{t_2}  \left\|\frac{\mathrm{d}  U^{(m)}   }{\mathrm{d}  t} \right\|_{-1} \mathrm{d}  t \right)^2 
			\leqslant \left|t_2 -t_1\right|  \int_{t_1}^{t_2}  \left\|\frac{\mathrm{d}  U^{(m)}   }{\mathrm{d}  t} \right\|_{-1}^2 \mathrm{d}  t  
			\\&\leqslant \left|t_2 -t_1\right|	\left\| \frac{\mathrm{d} U^{(m)}    }{\mathrm{d}  t}  \right\|_{L^{2}\left(0, T_m^* ;  \left(H^{-1}(\Omega)\right)^N\right)}^2.
		\end{aligned}
	\end{equation*}
	By the equivalence of norms in the finite-dimensional space $\mathcal{V}_m$, there exists a constant $C_m>0$, depending on $m$, such that $\| V \|_1 \leqslant C_m \| V \|_{-1}$ for all $V \in \left(\mathcal{V}_m\right)^N$. Taking the square root of the previous inequality yields
	\begin{equation*}
		\begin{aligned}
			\left\| U^{(m)}(t_2) -U^{(m)}(t_1)\right\|_1 \leqslant  C_m \sqrt{\left|t_2 -t_1\right|} \left\| \frac{\mathrm{d} U^{(m)}    }{\mathrm{d}  t}  \right\|_{L^{2}\left(0, T_m^* ;  \left(H^{-1}(\Omega)\right)^N\right)}.
		\end{aligned}
	\end{equation*}
	Since $\frac{\mathrm{d} U^{(m)}}{\mathrm{d} t} \in L^2\left(0, T_m^* ; \big(H^{-1}(\Omega)\big)^N\right)$ as established in Proposition \ref{existence of weak formulation of Galerkin equation}, this inequality demonstrates that $U^{(m)}(t)$ is uniformly continuous with respect to the $H_0^1$ norm. It guarantees that $U^{(m)}(t)$ is a Cauchy sequence as $t \rightarrow T_m^*$. Therefore, the strong limit exists
	\begin{equation*}
		\begin{aligned}
			\lim_{t \rightarrow T_m^*} U^{(m)}(t) = U^{(m)}_{T_m^*} \in \big(H_0^1(\Omega)\big)^N.
		\end{aligned}
	\end{equation*}
	Since $T_m^*$ is finite, we can use $U^{(m)}_{T_m^*}$ as an initial condition to extend the solution to an interval $[T_m^*, T_m^*+t_{\delta})$ for some $t_{\delta}>0$ via the Picard--Lindel\"{o}f theorem. This contradicts the assumed maximality of $T_m^*$, and thus we come to the conclusion that the solution exists for any given $T>0$, and
	\begin{equation*}
		\begin{aligned}
			U^{(m)} \in C\left([0, T) ;\big(H_0^1(\Omega)\big)^N\right).
		\end{aligned}
	\end{equation*}
\end{proof}

\subsection{Proof of Lemma~\ref{lem:compact embedding}}\label{proof of lemma: compact embedding}
\begin{proof}
	Let $\{u_n\}_{n=1}^{\infty}$ be bounded in $\mathcal D(\mathcal H)$. Then there exists \(M>0\) such that
	$$
	\|u_n\|_1\leqslant M,
	\qquad
	\|\mathcal Hu_n\|\leqslant M,
	\quad \forall n\geqslant1.
	$$
	Since the embedding \(H_0^1(\Omega)\hookrightarrow L^2(\Omega)\) is compact, there exists a subsequence, still denoted by \(\{u_n\}\), such that
	$$
	\|u_n-u_k\|\to0
	\quad\text{as }n,k\to\infty .
	$$
	For \(n,k\geqslant1\), applying G{\aa}rding's inequality to \(u_n-u_k\) gives
	$$
	c_1\|\nabla(u_n-u_k)\|^2
	\leqslant
	c_2\|u_n-u_k\|^2
	+
	\left\langle \mathcal H(u_n-u_k),u_n-u_k\right\rangle .
	$$
	Since \(u_n-u_k\in\mathcal D(\mathcal H)\), the last pairing is the \(L^2\)-inner product. Hence
	\begin{equation*}
		\begin{aligned}
			c_1\|\nabla(u_n-u_k)\|^2
		&	\leqslant
			c_2\|u_n-u_k\|^2
			+
			\|\mathcal H(u_n-u_k)\|\,\|u_n-u_k\|
		\\&	\leqslant
			c_2\|u_n-u_k\|^2+2M\|u_n-u_k\|.
		\end{aligned}
	\end{equation*}
	The right-hand side tends to zero as \(n,k\to\infty\). Thus
	$$
	\|\nabla(u_n-u_k)\|\to0.
	$$
	Together with \(\|u_n-u_k\|\to0\), this shows that \(\{u_n\}\) is Cauchy in \(H_0^1(\Omega)\). Since \(H_0^1(\Omega)\) is complete, this subsequence converges strongly in \(H_0^1(\Omega)\).
	
	For the product space, let \(U_n=(u_{n,1},\ldots,u_{n,N})\) be bounded in \((\mathcal D(\mathcal H))^N\). Applying the scalar result to each component and using a diagonal subsequence argument, we obtain a subsequence converging strongly in \((H_0^1(\Omega))^N\). Hence the embedding
	$$
	(\mathcal D(\mathcal H))^N\hookrightarrow (H_0^1(\Omega))^N
	$$
	is compact.
\end{proof}

\subsection{Proof of Lemma~\ref{solution of linear evolution problem}}
\label{proof of lemma:solution of linear evolution problem}

\begin{proof}
	Under the spectral assumptions on $\mathcal H$, the operator $\mathcal H$ is self-adjoint on $L^2(\Omega)$, with domain $\mathcal D(\mathcal H)$, and is bounded from below by $\lambda_1$.
	Hence $-k\mathcal H$ generates the strongly continuous semigroup
	$$
	\{\exp(-k\mathcal Ht)\}_{t\geqslant0}
	$$
	on $L^2(\Omega)$, and
	$$
	\|\exp(-k\mathcal Ht)\|
	\leqslant
	e^{-k\lambda_1t},
	\qquad t\geqslant0.
	$$
	The same assertion holds componentwise on $\big(L^2(\Omega)\big)^N$.
	
	For $U_0\in(\mathcal D(\mathcal H))^N$, define
	$$
	W(t)=\exp(-k\mathcal Ht)U_0.
	$$
	By the spectral theorem,
	$$
	W(t)\in(\mathcal D(\mathcal H))^N,
	\qquad
	\mathcal HW(t)=\exp(-k\mathcal Ht)\mathcal HU_0,
	\qquad t\geqslant0.
	$$
	Thus
	$$
	W\in C\left([0,\infty);\big(L^2(\Omega)\big)^N\right),
	\qquad
	\mathcal HW\in C\left([0,\infty);\big(L^2(\Omega)\big)^N\right).
	$$
	Hence
	$$
	W\in C\left([0,\infty);(\mathcal D(\mathcal H))^N\right),
	$$
	where $(\mathcal D(\mathcal H))^N$ is equipped with the graph norm. Moreover, the semigroup theorem gives
	$$
	\frac{\mathrm dW(t)}{\mathrm dt}
	=
	-k\mathcal HW(t),
	\qquad t\geqslant0,
	$$
	and therefore
	$$
	W\in C^1\left([0,\infty);\big(L^2(\Omega)\big)^N\right).
	$$
	Thus $W$ solves \eqref{linear evolution problem}. Uniqueness follows from the uniqueness of classical solutions generated by $-k\mathcal H$.
\end{proof}

\subsection{Proof of Lemma~\ref{negative condition of initial value}}	\label{proof of lemma: negative condition of initial value}
	\begin{proof}
		Let \(a\in\mathbb R^N\), and set
		$$
		v=U_0a.
		$$
		If \(v=0\), then
		$$
		a^\top U_0^\top e^{-k\mathcal Ht}U_0a
		=
		a^\top U_0^\top U_0a
		=
		0,
		$$
		and the desired inequality holds at $a$. Hence it suffices to consider the case \(v\neq0\).
		
		Since \(U_0\in(\mathcal D(\mathcal H))^N\), we have \(v\in\mathcal D(\mathcal H)\). Write the spectral expansion
		$$
		v=\sum_{j=1}^{\infty}\alpha_jv_j.
		$$
		Then
		$$
		\left(e^{-k\mathcal Ht}v,v\right)
		=
		\sum_{j=1}^{\infty}e^{-k\lambda_jt}|\alpha_j|^2,
		$$
		and
		$$
		\left(\mathcal Hv,v\right)
		=
		\sum_{j=1}^{\infty}\lambda_j|\alpha_j|^2.
		$$
		Since \(\lambda\mapsto e^{-k\lambda t}\) is convex with respect to \(t\geqslant0\), Jensen's inequality gives
		$$
		\begin{aligned}
			\frac{\left(e^{-k\mathcal Ht}v,v\right)}{\|v\|^2}
			=
			\frac{\sum_{j=1}^{\infty}e^{-k\lambda_jt}|\alpha_j|^2}
			{\sum_{j=1}^{\infty}|\alpha_j|^2}
			\geqslant
			\exp\left(
			-kt
			\frac{\sum_{j=1}^{\infty}\lambda_j|\alpha_j|^2}
			{\sum_{j=1}^{\infty}|\alpha_j|^2}
			\right)
		=
			\exp\left(
			-kt
			\frac{\left(\mathcal Hv,v\right)}{\|v\|^2}
			\right).
		\end{aligned}
		$$
		From \(U_0^\top\mathcal HU_0\leqslant0\), we obtain
		$$
		\left(\mathcal Hv,v\right)
		=
		a^\top U_0^\top\mathcal HU_0a
		\leqslant0.
		$$
		Consequently,
		$$
		\exp\left(
		-kt
		\frac{\left(\mathcal Hv,v\right)}{\|v\|^2}
		\right)
		\geqslant1,
		\qquad t\geqslant0.
		$$
		Thus
		$$
		\left(e^{-k\mathcal Ht}v,v\right)
		\geqslant
		\|v\|^2,
		\qquad t\geqslant0.
		$$
		Recalling that \(v=U_0a\), we arrive at
		$$
		a^\top U_0^\top e^{-k\mathcal Ht}U_0a
		\geqslant
		a^\top U_0^\top U_0a,
		\qquad \forall a\in\mathbb R^N.
		$$
		Therefore,
		$$
		U_0^\top e^{-k\mathcal Ht}U_0
		\geqslant
		U_0^\top U_0,
		\qquad t\geqslant0.
		$$
	\end{proof}

\subsection{Proof of Lemma~\ref{expression of Galerkin solution}}
\label{proof of lemma: expression of Galerkin solution}

\begin{proof}
	For $U^{(m)}(t)\in(\mathcal V_m)^N$, write
	$$
	U^{(m)}(t)=V_m\alpha^{(m)}(t),
	$$
	where
	$$
	V_m=(v_1,\ldots,v_m),
	\qquad
	V_m^\top V_m=I_m,
	$$
	and $\alpha^{(m)}(t)\in\mathbb R^{m\times N}$. Since
	$$
	\mathcal H_mV_m=V_m\Lambda_m,
	\qquad
	\Lambda_m=\operatorname{diag}(\lambda_1,\ldots,\lambda_m),
	$$
	substituting $U^{(m)}(t)=V_m\alpha^{(m)}(t)$ into the Galerkin equation leads to
	$$
	\frac{\mathrm d}{\mathrm dt}\alpha^{(m)}(t)
	=
	-\Lambda_m\alpha^{(m)}(t)
	+
	\alpha^{(m)}(t)
	\big(\alpha^{(m)}(t)\big)^\top
	\Lambda_m\alpha^{(m)}(t).
	$$
	 The finite-dimensional representation formula for Oja's flow \cite{yan1994global} then gives
	\begin{equation*}
		\begin{aligned}
			\alpha^{(m)}(t)
			=&
			\exp(-\Lambda_mt)\alpha^{(m)}(0)
			\left[
			I_N-\big(\alpha^{(m)}(0)\big)^\top\alpha^{(m)}(0) \right.
		\\&\quad	\left. + \big(\alpha^{(m)}(0)\big)^\top
			\exp(-2\Lambda_mt)\alpha^{(m)}(0)
			\right]^{-1/2}
			Q^{(m)}(t),
		\end{aligned}
	\end{equation*}
	where $Q^{(m)}(t)\in\mathcal O^N$.
	By Lemma~\ref{negative condition of initial value}, applied in the finite-dimensional Galerkin space,
	$$
	\big(\alpha^{(m)}(0)\big)^\top
	\exp(-2\Lambda_mt)\alpha^{(m)}(0)
	\geqslant
	\big(\alpha^{(m)}(0)\big)^\top\alpha^{(m)}(0),
	\qquad t\geqslant0.
	$$
	Thus
	$$
	I_N-\big(\alpha^{(m)}(0)\big)^\top\alpha^{(m)}(0)
	+
	\big(\alpha^{(m)}(0)\big)^\top
	\exp(-2\Lambda_mt)\alpha^{(m)}(0)
	\geqslant I_N.
	$$
	Hence $	\left[I_N-\big(\alpha^{(m)}(0)\big)^\top\alpha^{(m)}(0)+	\big(\alpha^{(m)}(0)\big)^\top	\exp(-2\Lambda_mt)\alpha^{(m)}(0)\right]^{-\frac{1}{2}}$ is well defined for all \(t\geqslant0\).
	
	Multiplying the representation of \(\alpha^{(m)}(t)\) by \(V_m\) from the left and using
	$$
	V_m\exp(-\Lambda_mt)\alpha^{(m)}(0)
	=
	\exp(-\mathcal H_mt)U^{(m)}(0),
	$$
	we obtain
	\begin{equation*}
		\begin{aligned}
			U^{(m)}(t)
			=&\,
			\exp(-\mathcal H_mt)U^{(m)}(0)
			\left[
			I_N-\big(U^{(m)}(0)\big)^\top U^{(m)}(0)
			\right.
			\\
			&\qquad\left.
			+
			\big(U^{(m)}(0)\big)^\top
			\exp(-2\mathcal H_mt)U^{(m)}(0)
			\right]^{-1/2}
			Q^{(m)}(t).
		\end{aligned}
	\end{equation*}
\end{proof}

\subsection{Proof of Lemma~\ref{singular values of solution}}
\label{proof of lemma: singular values of solution}

\begin{proof}
	Since $U_0^\top\mathcal{H}U_0<0$, the matrix $U_0^\top U_0$ is positive definite. Hence $U(t)^\top U(t)$ has $N$ positive eigenvalues for all $t\geqslant0$.

	We first prove the convergence of these eigenvalues. Since $U_0^\top\mathcal{H}U_0<0$, there exists a constant $c>0$ such that
	$$
	a^\top U_0^\top\mathcal{H}U_0a
	\leqslant
	-c\,a^\top U_0^\top U_0a,
	\qquad \forall a\in\mathbb R^N.
	$$
	By the argument in the proof of Lemma~\ref{negative condition of initial value}, for any $a\in\mathbb R^N$ and $t\geqslant0$,
	$$
	\begin{aligned}
		a^\top U_0^\top\exp(-2\mathcal{H}t)U_0a
		&\geqslant
		\exp\left(
		-2t
		\frac{a^\top U_0^\top\mathcal{H}U_0a}
		{a^\top U_0^\top U_0a}
		\right)
		a^\top U_0^\top U_0a
		\\
		&\geqslant
		e^{2ct}a^\top U_0^\top U_0a .
	\end{aligned}
	$$
	Thus
	$$
	\left(U_0^\top\exp(-2\mathcal{H}t)U_0\right)^{-1}
	\to0
	\qquad\text{as }t\to\infty.
	$$
	
	By Theorem~\ref{expression of solution(linear operator)},
	$$
	U(t)
	=
	\exp(-\mathcal{H}t)U_0
	\left[
	I_N-U_0^\top U_0
	+
	U_0^\top\exp(-2\mathcal{H}t)U_0
	\right]^{-1/2}
	Q(t),
	$$
	where $Q(t)\in\mathcal O^N$. Therefore, the eigenvalues of $U(t)^\top U(t)$ are the eigenvalues of
	$$
	\begin{aligned}
		&
		\left[
		I_N-U_0^\top U_0
		+
		U_0^\top\exp(-2\mathcal{H}t)U_0
		\right]^{-1/2}
		U_0^\top\exp(-2\mathcal{H}t)U_0
		\\
		&\qquad\cdot
		\left[
		I_N-U_0^\top U_0
		+
		U_0^\top\exp(-2\mathcal{H}t)U_0
		\right]^{-1/2}.
	\end{aligned}
	$$
	Equivalently, they are the eigenvalues of
	$$
	\begin{aligned}
		\left[
		I_N
		+
		\left(U_0^\top\exp(-2\mathcal{H}t)U_0\right)^{-1/2}
		\left(I_N-U_0^\top U_0\right)
		\left(U_0^\top\exp(-2\mathcal{H}t)U_0\right)^{-1/2}
		\right]^{-1}.
	\end{aligned}
	$$
	Since
	$$
	\left(U_0^\top\exp(-2\mathcal{H}t)U_0\right)^{-1}
	\to0,
	$$
	we obtain
	$$
	\lim_{t\to\infty}\sigma_i(t)=1,
	\qquad 1\leqslant i\leqslant N.
	$$
	
	We next prove the monotonicity assertions. For any $a\in\mathbb R^N\setminus\{0\}$, write
	$$
	U_0a=\sum_{j=1}^{\infty}\alpha_jv_j.
	$$
	Then
	$$
	\frac{
		\left(\mathcal{H}\exp(-\mathcal{H}t)U_0a,\exp(-\mathcal{H}t)U_0a\right)
	}
	{\|\exp(-\mathcal{H}t)U_0a\|^2}
	=
	\frac{
		\sum_{j=1}^{\infty}\lambda_j e^{-2\lambda_jt}|\alpha_j|^2
	}
	{
		\sum_{j=1}^{\infty}e^{-2\lambda_jt}|\alpha_j|^2
	}.
	$$
	Differentiating the last quotient gives
	$$
	\begin{aligned}
		&\frac{\mathrm d}{\mathrm dt}
		\frac{
			\sum_{j=1}^{\infty}\lambda_j e^{-2\lambda_jt}|\alpha_j|^2
		}
		{
			\sum_{j=1}^{\infty}e^{-2\lambda_jt}|\alpha_j|^2
		}
		\\
		&\quad =
		-2
		\frac{
			\left(\sum_{j=1}^{\infty}\lambda_j^2e^{-2\lambda_jt}|\alpha_j|^2\right)
			\left(\sum_{j=1}^{\infty}e^{-2\lambda_jt}|\alpha_j|^2\right)
			-
			\left(\sum_{j=1}^{\infty}\lambda_je^{-2\lambda_jt}|\alpha_j|^2\right)^2
		}
		{
			\left(\sum_{j=1}^{\infty}e^{-2\lambda_jt}|\alpha_j|^2\right)^2
		}
		\leqslant0.
	\end{aligned}
	$$
	Since $U_0^\top\mathcal{H}U_0<0$, the quotient is negative at $t=0$. Hence it is negative for all $t\geqslant0$. Therefore
	$$
	\begin{aligned}
		\frac{\mathrm d}{\mathrm dt}
		a^\top U_0^\top\exp(-2\mathcal{H}t)U_0a
		&=
		-2
		\left(\mathcal{H}\exp(-\mathcal{H}t)U_0a,
		\exp(-\mathcal{H}t)U_0a\right)
		\\
		&\geqslant0.
	\end{aligned}
	$$
	Thus $U_0^\top\exp(-2\mathcal{H}t)U_0$ is nondecreasing, and its inverse is nonincreasing.
	
	If $\sigma_i(0)\leqslant1$ for $1\leqslant i\leqslant N$, then $U_0^\top U_0\leqslant I_N$. The matrix
	$$
	\begin{aligned}
		&\left(U_0^\top\exp(-2\mathcal{H}t)U_0\right)^{-1/2}
		\left(I_N-U_0^\top U_0\right)
		\left(U_0^\top\exp(-2\mathcal{H}t)U_0\right)^{-1/2}
	\end{aligned}
	$$
	has the same eigenvalues as
	$$
	\begin{aligned}
		&\left(I_N-U_0^\top U_0\right)^{1/2}
		\left(U_0^\top\exp(-2\mathcal{H}t)U_0\right)^{-1}
		\left(I_N-U_0^\top U_0\right)^{1/2},
	\end{aligned}
	$$
	which is nonincreasing. Hence the ordered eigenvalues of
	$$
	\begin{aligned}
		&\left[
		I_N
		+
		\left(U_0^\top\exp(-2\mathcal{H}t)U_0\right)^{-1/2}
		\left(I_N-U_0^\top U_0\right)
		\left(U_0^\top\exp(-2\mathcal{H}t)U_0\right)^{-1/2}
		\right]^{-1}
	\end{aligned}
	$$
	are nondecreasing. Since $I_N-U_0^\top U_0\geqslant0$, they are bounded above by $1$. Hence $\sigma_i(t)$ is nondecreasing and satisfies
	$$
	\sigma_i(t)\leqslant1,
	\qquad 1\leqslant i\leqslant N.
	$$
	
	If $\sigma_i(0)\geqslant1$ for $1\leqslant i\leqslant N$, then $U_0^\top U_0\geqslant I_N$. The same argument, applied to $U_0^\top U_0-I_N$, shows that the ordered eigenvalues of
	$$
	\begin{aligned}
		&\left[
		I_N
		-
		\left(U_0^\top\exp(-2\mathcal{H}t)U_0\right)^{-1/2}
		\left(U_0^\top U_0-I_N\right)
		\left(U_0^\top\exp(-2\mathcal{H}t)U_0\right)^{-1/2}
		\right]^{-1}
	\end{aligned}
	$$
	are nonincreasing. Since
	$$
	I_N-U_0^\top U_0+U_0^\top\exp(-2\mathcal{H}t)U_0\geqslant I_N,
	$$
	these eigenvalues are bounded below by $1$. Hence $\sigma_i(t)$ is nonincreasing and satisfies
	$$
	\sigma_i(t)\geqslant1,
	\qquad 1\leqslant i\leqslant N.
	$$
	
	Finally, let
	$$
	\beta=\min\{1,\sigma_1(0)\}.
	$$
	Since
	$$
	U_0^\top U_0\geqslant\beta I_N
	$$
	and
	$$
	U_0^\top\exp(-2\mathcal{H}t)U_0\geqslant U_0^\top U_0,
	$$
	we obtain
	$$
	\begin{aligned}
		&U_0^\top\exp(-2\mathcal{H}t)U_0
		-
		\beta
		\left[
		I_N-U_0^\top U_0
		+
		U_0^\top\exp(-2\mathcal{H}t)U_0
		\right]
		\\
		&\quad =
		(1-\beta)U_0^\top\exp(-2\mathcal{H}t)U_0
		+
		\beta U_0^\top U_0
		-
		\beta I_N
		\\
		&\quad \geqslant
		U_0^\top U_0-\beta I_N
		\geqslant0.
	\end{aligned}
	$$
	Therefore, every eigenvalue of $U(t)^\top U(t)$ is bounded below by $\beta$, that is,
	$$
	\sigma_i(t)\geqslant\beta,
	\qquad 1\leqslant i\leqslant N,\quad t\geqslant0.
	$$
\end{proof}

\subsection{Proof of Lemma~\ref{Bound of UHU}}
\label{proof of lemma: Bound of UHU}

\begin{proof}
By Lemma~\ref{singular values of solution}, $U_0^\top U_0$ is positive definite. Define
$$
c_*=-\max_{a\in\mathbb R^N\backslash\{0\}}\frac{a^\top U_0^\top\mathcal H U_0a}{a^\top U_0^\top U_0a}.
$$
Thus, for every $a\in\mathbb R^N$ and $t\geqslant0$,
$$
(\mathcal H\exp(-\mathcal Ht)U_0a,\exp(-\mathcal Ht)U_0a)\leqslant -c_*\|\exp(-\mathcal Ht)U_0a\|^2.
$$
	
	By Theorem~\ref{expression of solution(linear operator)}, there exists $Q(t)\in\mathcal O^N$ such that
	$$
	U(t)
	=
	\exp(-\mathcal Ht)U_0
	\left[
	I_N-U_0^\top U_0
	+
	U_0^\top\exp(-2\mathcal Ht)U_0
	\right]^{-1/2}
	Q(t).
	$$
	For any $b\in\mathbb R^N$, set
	$$
	a
	=
	\left[
	I_N-U_0^\top U_0
	+
	U_0^\top\exp(-2\mathcal Ht)U_0
	\right]^{-1/2}
	Q(t)b.
	$$
	Then
	$$
	U(t)b=\exp(-\mathcal Ht)U_0a.
	$$
	Using the preceding estimate, we obtain
	\begin{equation*}
		\begin{aligned}
			b^\top U(t)^\top\mathcal H U(t)b
			&=
			\left(\mathcal H U(t)b,U(t)b\right)
			\\
			&=
			\left(
			\mathcal H\exp(-\mathcal Ht)U_0a,
			\exp(-\mathcal Ht)U_0a
			\right)
			\\
			&\leqslant
			-c_*
			\left\|\exp(-\mathcal Ht)U_0a\right\|^2
			\\
			&=
			-c_*\,b^\top U(t)^\top U(t)b .
		\end{aligned}
	\end{equation*}
	By Lemma~\ref{singular values of solution}, with
	$$
	\beta=\min\{1,\sigma_1(0)\}>0,
	$$
	we have
	$$
	U(t)^\top U(t)\geqslant \beta I_N,
	\qquad t\geqslant0.
	$$
	Therefore
	$$
	b^\top U(t)^\top\mathcal H U(t)b
	\leqslant
	-c_*\beta |b|^2,
	\qquad b\in\mathbb R^N,\quad t\geqslant0.
	$$
	Let
	$$
	c_0=c_*\beta .
	$$
	Then $c_0>0$ depends only on $U_0$, and
	$$
	U(t)^\top\mathcal H U(t)
	\leqslant
	-c_0I_N,
	\qquad t\geqslant0.
	$$
	In particular,
	$$
	U(t)^\top\mathcal H U(t)<0,
	\qquad t\geqslant0,
	$$
	which together with Lemma \ref{lem:compact embedding} completes the proof.
	\end{proof}

\subsection{Proof of Lemma~\ref{convergence of Z}}\label{proof of lemma: convergence of Z}
\begin{proof}
	Under the spectral assumptions on $\mathcal H$, the eigenvalues satisfy $\lambda_i\to+\infty$.
	 Hence $\mathcal H$ has only finitely many strictly negative eigenvalues. Let
	$$
	\lambda_1\leqslant\lambda_2\leqslant\cdots\leqslant\lambda_L<0
	$$
	be all strictly negative eigenvalues of $\mathcal H$. Let $\mathcal P_{\lambda<0}$ be the spectral projection associated with these eigenvalues, and set
	$$
	\mathcal P_{\lambda\geqslant0}=\mathcal I-\mathcal P_{\lambda<0}.
	$$

	If $\mathcal P_{\lambda<0}U_0a=0$ for some $a\in\mathbb R^N$, then $U_0a=\mathcal P_{\lambda\geqslant0}U_0a$. By the spectral decomposition, $\mathcal H$ is nonnegative on the nonnegative spectral subspace. Thus
	$$
	a^\top U_0^\top\mathcal H U_0a
	=
	\langle\mathcal H U_0a,U_0a\rangle
	\geqslant0.
	$$
	Since $U_0^\top\mathcal H U_0<0$, this gives $a=0$. Hence $\mathcal P_{\lambda<0}U_0$ has full column rank $N$, and $L\geqslant N$.
	
	Set
	$$
	W(t)=e^{-\mathcal Ht}U_0,\qquad
	W_-(t)=e^{-\mathcal Ht}\mathcal P_{\lambda<0}U_0,\qquad
	W_+(t)=e^{-\mathcal Ht}\mathcal P_{\lambda\geqslant0}U_0.
	$$
	Write
	$$
	\mathcal P_{\lambda<0}U_0
	=
	\left(
	\sum_{j=1}^L m_{j1}v_j,\ldots,
	\sum_{j=1}^L m_{jN}v_j
	\right),
	$$
	and let $M=(m_{jr})\in\mathbb R^{L\times N}$. Since $M$ has rank $N$, choose rows $M_{i_1},\ldots,M_{i_N}$ such that
	$$
	M_I=(M_{i_1},\ldots,M_{i_N})^\top
	$$
	is nonsingular and
	$$
	\sum_{r=1}^N\lambda_{i_r}
	$$
	is minimal among all nonsingular choices of $N$ rows. Set
	$$
	I=\{i_1,\ldots,i_N\},
	\qquad
	\Lambda_I=\operatorname{diag}(\lambda_{i_1},\ldots,\lambda_{i_N}).
	$$
	For each $t\geqslant0$, the matrix $M_I^{-1}e^{\Lambda_I t}$ is nonsingular. Hence the columns of $W_-(t)$ and $W_-(t)M_I^{-1}e^{\Lambda_I t}$ generate the same subspace. The $r$-th column of $W_-(t)M_I^{-1}e^{\Lambda_I t}$ is
	$$
	v_{i_r}
	+
	\sum_{\substack{1\leqslant j\leqslant L\\ j\notin I}}
	(M_jM_I^{-1})_r e^{-(\lambda_j-\lambda_{i_r})t}v_j,
	$$
	where $M_j$ is the $j$-th row of $M$. If $j\notin I$, $\lambda_j<\lambda_{i_r}$, and $(M_jM_I^{-1})_r\neq0$, then replacing the row $M_{i_r}$ in $M_I$ by $M_j$ gives a nonsingular minor with a smaller eigenvalue sum. This contradicts the choice of $I$. Therefore
	$$
	(M_jM_I^{-1})_r=0
	\qquad
	\text{whenever }j\notin I\text{ and }\lambda_j<\lambda_{i_r}.
	$$
	It follows that the $r$-th column of $W_-(t)M_I^{-1}e^{\Lambda_I t}$ converges exponentially in the graph norm of $\mathcal H$ to
	$$
	\phi_r^\infty
	=
	v_{i_r}
	+
	\sum_{\substack{1\leqslant j\leqslant L\\ j\notin I,\ \lambda_j=\lambda_{i_r}}}
	(M_jM_I^{-1})_r v_j,
	\qquad r=1,\ldots,N.
	$$
	The functions $\phi_1^\infty,\ldots,\phi_N^\infty$ are linearly independent, since their coefficients in the rows indexed by $I$ form the identity matrix.
	
	Define
	$$
	\mathcal Z_\infty V
	=
	\scalebox{0.94}{$(\phi_1^\infty,\ldots,\phi_N^\infty)
	\left[
	(\phi_1^\infty,\ldots,\phi_N^\infty)^\top
	(\phi_1^\infty,\ldots,\phi_N^\infty)
	\right]^{-1}
	(\phi_1^\infty,\ldots,\phi_N^\infty)^\top V, \ 
	\forall V\in(L^2(\Omega))^N.$} 
	$$
	Set
	$$
	G_-(t)=W_-(t)^\top W_-(t).
	$$
	Since $M_I^{-1}e^{\Lambda_I t}$ is nonsingular,
	\begin{equation*}
		\begin{aligned}
			W_-(t)G_-(t)^{-1}W_-(t)^\top
			={}&
			W_-(t)M_I^{-1}e^{\Lambda_I t}
		 \cdot
			\left[
			\big(W_-(t)M_I^{-1}e^{\Lambda_I t}\big)^\top
			\big(W_-(t)M_I^{-1}e^{\Lambda_I t}\big)
			\right]^{-1}
			\\
			&\cdot
			\big(W_-(t)M_I^{-1}e^{\Lambda_I t}\big)^\top .
		\end{aligned}
	\end{equation*}
	Therefore, there holds
	$$
\scalebox{0.94}{$ 		\left\|
	\left[
	W_-(t)G_-(t)^{-1}W_-(t)^\top-\mathcal Z_\infty
	\right]V
	\right\|
	+
	\left\|
	\mathcal H
	\left[
	W_-(t)G_-(t)^{-1}W_-(t)^\top-\mathcal Z_\infty
	\right]V
	\right\|
	\leqslant
	Ce^{-\gamma t}\|V\| $}
	$$
	for all $t\geqslant0$ and all $V\in(L^2(\Omega))^N$.
	
	By Theorem~\ref{expression of solution(linear operator)},
	$$
	\mathcal Z(t)
	=
	W(t)
	\left[I_N-U_0^\top U_0+W(t)^\top W(t)\right]^{-1}
	W(t)^\top .
	$$
	The negative and nonnegative spectral subspaces are orthogonal, and hence
	$$
	W(t)^\top W(t)=G_-(t)+W_+(t)^\top W_+(t).
	$$
	Note that $\mathcal P_{\lambda<0}U_0$ has full column rank, we have
	$$
	G_-(t)\geqslant ce^{-2\lambda_Lt}I_N,
	\qquad
	\|G_-(t)^{-1/2}\|\leqslant Ce^{\lambda_Lt}.
	$$
	Moreover, $W_-(t)G_-(t)^{-1/2}$ has orthonormal columns and lies in the finite-dimensional negative spectral subspace. Thus
	$$
	\|\mathcal HW_-(t)G_-(t)^{-1/2}\|\leqslant C.
	$$
	By the spectral theorem and Lemma~\ref{solution of linear evolution problem},
	$$
	\mathcal HW_+(t)
	=
	e^{-\mathcal Ht}\mathcal H\mathcal P_{\lambda\geqslant0}U_0.
	$$
	Since $e^{-\mathcal Ht}$ is contractive on the nonnegative spectral subspace, $W_+(t)$ and $\mathcal HW_+(t)$ are uniformly bounded. Therefore
	$$
	\|W_+(t)G_-(t)^{-1/2}\|
	+
	\|\mathcal HW_+(t)G_-(t)^{-1/2}\|
	\leqslant
	Ce^{\lambda_Lt}.
	$$
	
	Let
	$$
	\Delta(t)
	=
	G_-(t)^{-1/2}
	\left[I_N-U_0^\top U_0+W_+(t)^\top W_+(t)\right]
	G_-(t)^{-1/2}.
	$$
	Then
	$$
	\|\Delta(t)\|\leqslant Ce^{2\lambda_Lt}.
	$$
	For all sufficiently large $t$,
	$$
	\|[I_N+\Delta(t)]^{-1}\|\leqslant C,
	\qquad
	\|[I_N+\Delta(t)]^{-1}-I_N\|\leqslant Ce^{2\lambda_Lt},
	$$
	and
	\begin{equation*}
		\begin{aligned}
			&\left[I_N-U_0^\top U_0+W(t)^\top W(t)\right]^{-1}
			\\
			&\quad =
			G_-(t)^{-1/2}[I_N+\Delta(t)]^{-1}G_-(t)^{-1/2}.
		\end{aligned}
	\end{equation*}
	Using $W(t)=W_-(t)+W_+(t)$, the preceding bounds yield
	$$
\scalebox{0.93}{$ 		\left\|
	\left[
	\mathcal Z(t)-W_-(t)G_-(t)^{-1}W_-(t)^\top
	\right]V
	\right\|
	+
	\left\|
	\mathcal H
	\left[
	\mathcal Z(t)-W_-(t)G_-(t)^{-1}W_-(t)^\top
	\right]V
	\right\|
	\leqslant
	Ce^{\lambda_Lt}\|V\| $}
	$$
	for all sufficiently large $t$ and all $V\in(L^2(\Omega))^N$. 
	
	Combining the last two estimates, we obtain
	$$
	\|(\mathcal Z(t)-\mathcal Z_\infty)V\|
	+
	\|\mathcal H(\mathcal Z(t)-\mathcal Z_\infty)V\|
	\leqslant
	Ke^{-\gamma t}\|V\|,
	\qquad
	\forall t\geqslant0,\quad
	\forall V\in(L^2(\Omega))^N.
	$$
	
	By construction,
	$$
	\mathcal Z_\infty^2=\mathcal Z_\infty.
	$$
	Each $\phi_r^\infty$ belongs to $\ker(\mathcal H-\lambda_{i_r}I)$. Hence the range of $\mathcal Z_\infty$ and its orthogonal complement are invariant under $\mathcal H$. Thus $\mathcal Z_\infty$ commutes with $\mathcal H$ on $(\mathcal D(\mathcal H))^N$, that is,
	$$
	\mathcal H\mathcal Z_\infty V
	=
	\mathcal Z_\infty\mathcal H V,
	\qquad
	\forall V\in(\mathcal D(\mathcal H))^N.
	$$
	Furthermore,
	$$
	\mathcal H\mathcal Z_\infty V
	=
	\mathcal Z_\infty\mathcal H\mathcal Z_\infty V,
	\qquad
	\forall V\in(L^2(\Omega))^N.
	$$
\end{proof}

\subsection{Proof of Lemma~\ref{meaning of initial condition}}\label{proof of lemma: meaning of initial condition} 
\begin{proof}
	If $[U]\in B([V^{(*)}],\delta)$, then there exists $Q\in\mathcal{O}^N$ such that 
	\begin{equation*}
		\|UQ-V^{(*)}\|_1\leqslant \delta.
	\end{equation*}
	Thus we obtain
	\begin{equation*}
		\|(V^{(*)})^\top UQ-I_N\| \leqslant \|UQ-V^{(*)}\| \leqslant \delta.
	\end{equation*}
	For $\delta<1$, the matrix $(V^{(*)})^\top UQ$ is invertible. 
	Hence, we conclude that
	\begin{equation*}
		\operatorname{rank}\big((V^{(*)})^\top U\big)=N.
	\end{equation*}
	Equivalently, the projection of $\operatorname{span}\{U\}$ onto the target eigenspace $\operatorname{span}\{v_1,\ldots,v_N\}$ has full rank. 
	This guarantees that $U$ contains components of all eigenfunctions corresponding to the $N$ smallest eigenvalues.
\end{proof}

\subsection{Proof of Corollary~\ref{convergence rate}}\label{proof of corollary: convergence rate}
\begin{proof}
	Differentiating gives
	\begin{equation}
		\label{d nabla_GE(U)/dt}
		\begin{aligned}
			\frac{1}{2}\frac{\mathrm{d}}{\mathrm{d}t}\|\nabla_GE(U(t))\|^2
			&=\left\langle \nabla^2E(U(t))\frac{\mathrm{d}U(t)}{\mathrm{d}t}-\frac{\mathrm{d}U(t)}{\mathrm{d}t}U(t)^\top\nabla E(U(t)),\nabla_GE(U(t))\right\rangle
			\\
			&\quad-\left\langle U(t)\frac{\mathrm{d}}{\mathrm{d}t}\left(U(t)^\top\nabla E(U(t))\right),\nabla_GE(U(t))\right\rangle .
		\end{aligned}
	\end{equation}
	
	For the first term on the right-hand side of \eqref{d nabla_GE(U)/dt}, we have
	\begin{equation*}
		\begin{aligned}
			&\left\langle \nabla^2E(U(t))\frac{\mathrm{d}U(t)}{\mathrm{d}t}-\frac{\mathrm{d}U(t)}{\mathrm{d}t}U(t)^\top\nabla E(U(t)),\nabla_GE(U(t))\right\rangle
			\\
			&= \scalebox{0.97}{$ 	-\left\langle \nabla^2E(U(t))\nabla_GE(U(t)),\nabla_GE(U(t))\right\rangle+\left\langle \nabla_GE(U(t))U(t)^\top\nabla E(U(t)),\nabla_GE(U(t))\right\rangle $}
			\\
			&=-\operatorname{Hess}_{\tilde G}E(U(t))\left[\nabla_GE(U(t)),\nabla_GE(U(t))\right]
			\\
			&\leqslant -\mu\|\nabla_GE(U(t))\|^2 .
		\end{aligned}
	\end{equation*}
	We decompose the second term as
	\begin{equation*}
		\begin{aligned}
			&-\left\langle U(t)\frac{\mathrm{d}}{\mathrm{d}t}\left(U(t)^\top\nabla E(U(t))\right),\nabla_GE(U(t))\right\rangle
			\\
			&=-\left\langle U(t)\frac{\mathrm{d}}{\mathrm{d}t}\left(U(t)^\top\nabla E(U(t))\right),\nabla E(U(t))U(t)^\top U(t)-U(t)U(t)^\top\nabla E(U(t))\right\rangle
			\\
			&\quad-\left\langle U(t)\frac{\mathrm{d}}{\mathrm{d}t}\left(U(t)^\top\nabla E(U(t))\right),\nabla E(U(t))\left(I_N-U(t)^\top U(t)\right)\right\rangle .
		\end{aligned}
	\end{equation*}
	Since $U^\top\left(\nabla E(U)U^\top U-UU^\top\nabla E(U)\right)=-\left(\nabla E(U)U^\top U-UU^\top\nabla E(U)\right)^\top U$
	is skew-symmetric and $\frac{\mathrm{d}}{\mathrm{d}t}\left(U(t)^\top\nabla E(U(t))\right)$ is symmetric, we obtain
	\begin{equation*}
		\begin{aligned}
			&-\left\langle U(t)\frac{\mathrm d}{\mathrm dt}\left(U(t)^\top\nabla E(U(t))\right),\nabla_GE(U(t))\right\rangle
			\\
			&=-\left\langle U(t)\frac{\mathrm{d}}{\mathrm{d}t}\left(U(t)^\top\nabla E(U(t))\right),\nabla E(U(t))\left(I_N-U(t)^\top U(t)\right)\right\rangle
			\\
			&=-\operatorname{tr}\left[U(t)^\top\mathcal H U(t)\left(I_N-U(t)^\top U(t)\right)\left(\left(\frac{\mathrm{d}U(t)}{\mathrm{d}t}\right)^\top\mathcal H U(t)+U(t)^\top\mathcal H\frac{\mathrm{d}U(t)}{\mathrm{d}t}\right)\right]
			\\
			&=\operatorname{tr}\left[U(t)^\top\mathcal H U(t)\left(I_N-U(t)^\top U(t)\right)\left(\nabla_GE(U(t))^\top\mathcal H U(t)+U(t)^\top\mathcal H\nabla_GE(U(t))\right)\right].
		\end{aligned}
	\end{equation*}
	By Lemma~\ref{singular values of solution},
	$$
	0\leqslant I_N-U(t)^\top U(t)\leqslant I_N,\qquad t\geqslant0 .
	$$
	Note that $0\leqslant U(t)^\top U(t)\leqslant I_N$, while $U(t)^\top\mathcal H U(t)\leqslant0$, we have
	$$
	\lambda_1I_N\leqslant U(t)^\top\mathcal H U(t)\leqslant0,\qquad \|U(t)^\top\mathcal H U(t)\|\leqslant |\lambda_1|.
	$$
	Moreover,
	\begin{equation*}
		\begin{aligned}
		&	\nabla_GE(U(t))^\top\nabla_GE(U(t))
			\\&=\left(\mathcal H U(t)-U(t)\left(U(t)^\top\mathcal H U(t)\right)\right)^\top\left(\mathcal H U(t)-U(t)\left(U(t)^\top\mathcal H U(t)\right)\right)
			\\ &=(\mathcal H U(t))^\top\mathcal H U(t)-\left(U(t)^\top\mathcal H U(t)\right)^2-U(t)^\top\mathcal H U(t)\left(I_N-U(t)^\top U(t)\right)U(t)^\top\mathcal H U(t).
		\end{aligned}
	\end{equation*}
	Hence, with
	$$
	M_G=\sup_{t\geqslant0}\|\nabla_GE(U(t))\|<\infty,
	$$
	we obtain
	$$
	\|(\mathcal H U(t))^\top\mathcal H U(t)\|\leqslant M_G^2+2|\lambda_1|^2 .
	$$
	Furthermore,
	\begin{equation*}
		\begin{aligned}
			&\operatorname{tr}\left[U(t)^\top\mathcal H U(t)\left(I_N-U(t)^\top U(t)\right)\left(\nabla_GE(U(t))^\top\mathcal H U(t)+U(t)^\top\mathcal H\nabla_GE(U(t))\right)\right]
			\\
			&=2\operatorname{tr}\left[U(t)^\top\mathcal H U(t)\left(I_N-U(t)^\top U(t)\right)\left((\mathcal H U(t))^\top\mathcal H U(t)-\left(U(t)^\top\mathcal H U(t)\right)^2\right)\right]
			\\
			&\leqslant 2N|\lambda_1|\left(M_G^2+3|\lambda_1|^2\right)\left\|I_N-U(t)^\top U(t)\right\|.
		\end{aligned}
	\end{equation*}
	Denote
	$$
	C^*=2N|\lambda_1|\left(M_G^2+3|\lambda_1|^2\right).
	$$
	Combining the preceding estimates with Theorem~\ref{limiting solution in stiefel manifold} gives
	\begin{equation*}
		\begin{aligned}
			\frac{1}{2}\frac{\mathrm{d}}{\mathrm{d}t}\left\|\nabla_GE(U(t))\right\|^2
			&\leqslant -\mu\left\|\nabla_GE(U(t))\right\|^2+C^*\left\|I_N-U(t)^\top U(t)\right\|
			\\
			&\leqslant -\mu\left\|\nabla_GE(U(t))\right\|^2+C^*\left\|I_N-U_0^\top U_0\right\|e^{-2c_0t}.
		\end{aligned}
	\end{equation*}
	Applying Gr\"{o}nwall's inequality yields
	\begin{equation*}
		\begin{aligned}
		&	\left\|\nabla_GE(U(t))\right\|^2
			\\&		\leqslant
	\begin{cases}
				\left\|\nabla_GE(U_0)\right\|^2 e^{-2\mu t}+2C^*\left\|I_N-U_0^\top U_0\right\|te^{-2\mu t}, & \mu=c_0,\\[2mm]
				\left\|\nabla_GE(U_0)\right\|^2 e^{-2\mu t}+\dfrac{C^*\left\|I_N-U_0^\top U_0\right\|}{\mu-c_0}\left(e^{(2\mu-2c_0)t}-1\right)e^{-2\mu t}, & \mu\neq c_0.
			\end{cases}
		\end{aligned}
	\end{equation*}
	The proof is now complete.
\end{proof}

\bibliographystyle{siamplain}
\bibliography{reference}

@article{HiriartUrrutyYe1995,
	author = {Hiriart-Urruty, J.-B. and Ye, D.},
	title = {Sensitivity analysis of all eigenvalues of a symmetric matrix},
	journal = {Numerische Mathematik},
	volume = {70},
	year = {1995},
	pages = {45--72}
}

@article{MazzoleniSavare2023,
	author = {Mazzoleni, D. and Savare, G.},
	title = {${L}^2$-gradient flows of spectral functionals},
	journal = {Discrete and Continuous Dynamical Systems},
	volume = {43},
	number = {3--4},
	year = {2023},
	pages = {1560--1594}
}

@article{wang2026quasi,
	title={A quasi-orthogonal iterative method for eigenvalue problems},
	author={Wang, Shengyue and Zhou, Aihui},
	journal={arXiv preprint arXiv:2601.02108},
	year={2026}
}

@article{chu2025orthogonality,
	title={An orthogonality-preserving approach for eigenvalue problems},
	author={Chu, Tianyang and Dai, Xiaoying and Wang, Shengyue and Zhou, Aihui},
	journal={arXiv preprint arXiv:2511.06788},
	year={2025}
}

@article{chen2024convergence,
	title={On the convergence of Sobolev gradient flow for the {G}ross--{P}itaevskii eigenvalue problem},
	author={Chen, Ziang and Lu, Jianfeng and Lu, Yulong and Zhang, Xiangxiong},
	journal={SIAM J. Numer. Anal.},
	volume={62},
	number={2},
	pages={667--691},
	year={2024},
	publisher={SIAM}
}

@article{chen2024fully,
	title={Fully discretized Sobolev gradient flow for the {G}ross-{P}itaevskii eigenvalue problem},
	author={Chen, Ziang and Lu, Jianfeng and Lu, Yulong and Zhang, Xiangxiong},
	journal={Math. Comput.},
	volume={94},
	pages={2723--2760},
	year={2025}
}

@article{henning2020sobolev,
	title={Sobolev gradient flow for the Gross--Pitaevskii eigenvalue problem: Global convergence and computational efficiency},
	author={Henning, Patrick and Peterseim, Daniel},
	journal={SIAM J. Numer. Anal.},
	volume={58},
	number={3},
	pages={1744--1772},
	year={2020},
	publisher={SIAM}
}

@book{cances2023density,
	title = {Density Functional Theory: Modeling, Mathematical Analysis, Computational Methods, and Applications},
	editor = {Cancès, Eric and Friesecke, Gero},
	year = {2023},
	publisher = {Springer},
	address = {Berlin}
}

@article{chen2014adaptive,
	title = {Adaptive finite element approximations for {K}ohn--{S}ham models},
	author = {Chen, Huajie and Dai, Xiaoying and Gong, Xin and He, Lianhua and Zhou, Aihui},
	journal = {Multiscale Model. Simul.},
	volume = {12},
	pages = {1828--1869},
	year = {2014},
	publisher = {SIAM}
}

@article{chen2013numerical,
	title = {Numerical analysis of finite dimensional approximations of {K}ohn-{S}ham models},
	author = {Chen, Huajie and Gong, Xin and He, Lianhua and Yang, Zhi and Zhou, Aihui},
	journal = {Adv. Comput. Math.},
	volume = {38},
	pages = {225--256},
	year = {2013},
	publisher = {Springer}
}

@book{teschl2012ordinary,
	author    = {Teschl, Gerald},
	title     = {Ordinary Differential Equations and Dynamical Systems},
	series    = {Graduate Studies in Mathematics},
	volume    = {140},
	publisher = {American Mathematical Society},
	address   = {Providence, RI},
	year      = {2012}
}

@article{lebris2005computational,
	title = {Computational chemistry from the perspective of numerical analysis},
	author = {Le Bris, Claude},
	journal = {Acta Numer.},
	volume = {14},
	pages = {363--444},
	year = {2005},
	publisher = {Cambridge University Press}
}

@article{saad2010numerical,
	title = {Numerical methods for electronic structure calculations of materials},
	author = {Saad, Yousef and Chelikowsky, James R. and Shontz, Suzanne M.},
	journal = {SIAM Rev.},
	volume = {52},
	pages = {3--54},
	year = {2010},
	publisher = {SIAM}
}

@book{dorfler2011photonic,
	title={Photonic Crystals: Mathematical Analysis and Numerical Approximation},
	author={D{\"o}rfler, Willy and Lechleiter, Armin and Plum, Michael and Schneider, Guido and Wieners, Christian},
	volume={42},
	year={2011},
	publisher={Springer Science \& Business Media},
	address={Berlin}
}

@article{dirac1929quantum,
	title={Quantum mechanics of many-electron systems},
	author={Dirac, Paul Adrien Maurice},
	journal={Proc. R. Soc. Lond. A},
	volume={123},
	number={792},
	pages={714--733},
	year={1929},
	publisher={The Royal Society London}
}

@article{schwab2006karhunen,
	title={Karhunen--{L}o{\`e}ve approximation of random fields by generalized fast multipole methods},
	author={Schwab, Christoph and Todor, Radu Alexandru},
	journal={J. Comput. Phys.},
	volume={217},
	number={1},
	pages={100--122},
	year={2006},
	publisher={Elsevier}
}

@article{breuer1991use,
	title={The use of the {K}arhunen--{L}o{\`e}ve procedure for the calculation of linear eigenfunctions},
	author={Breuer, Kenneth S and Sirovich, Lawrence},
	journal={J. Comput. Phys.},
	volume={96},
	number={2},
	pages={277--296},
	year={1991},
	publisher={Elsevier}
}

@book{smith2024uncertainty,
	author={Smith, Ralph C.},
	title={Uncertainty Quantification: Theory, Implementation, and Applications},
	publisher={SIAM},
	year={2024},
	address={Philadelphia}
}

@book{sholl2009density,
	title     = {Density Functional Theory: A Practical Introduction},
	author    = {David S. Sholl and Janice A. Steckel},
	year      = {2009},
	publisher = {John Wiley \& Sons},
	address   = {Hoboken},
	isbn      = {978-0-470-37317-0}
}

@book{saad1992numerical,
	title     = {Numerical Methods for Large Eigenvalue Problems},
	author    = {Youcef Saad},
	year      = {1992},
	publisher = {SIAM},
	address   = {Philadelphia},
	series    = {Classics in Applied Mathematics},
	isbn      = {978-0-89871-534-7}
}

@book{golub2013matrix,
	title     = {Matrix Computations},
	author    = {Gene H. Golub and Charles F. Van Loan},
	year      = {2013},
	edition   = {4th},
	publisher = {Johns Hopkins University Press},
	address   = {Baltimore},
	isbn      = {978-1-4214-0794-4}
}

@book{martin2020electronic,
	author    = {R. M. Martin},
	title     = {Electronic Structure: Basic Theory and Practical Methods},
	publisher = {Cambridge University Press},
	year      = {2020},
	address   = {Cambridge}
}

@article{payne1992iterative,
	author    = {M. C. Payne and M. P. Teter and D. C. Allan and T. Arias and J. Joannopoulos},
	title     = {Iterative minimization techniques for ab initio total-energy calculations: molecular dynamics and conjugate gradients},
	journal   = {Rev. Mod. Phys.},
	volume    = {64},
	year      = {1992},
	pages     = {1045}
}

@article{oja1982simplified,
	title={Simplified neuron model as a principal component analyzer},
	author={Oja, Erkki},
	journal={J. Math. Biol.},
	volume={15},
	pages={267--273},
	year={1982},
	publisher={Springer}
}

@article{oja1989neural,
	title={Neural networks, principal components, and subspaces},
	author={Oja, Erkki},
	journal={Int. J. Neural Syst.},
	volume={1},
	number={01},
	pages={61--68},
	year={1989},
	publisher={World Scientific}
}

@article{oja1985stochastic,
	title={On stochastic approximation of the eigenvectors and eigenvalues of the expectation of a random matrix},
	author={Oja, Erkki and Karhunen, Juha},
	journal={J. Math. Anal. Appl.},
	volume={106},
	number={1},
	pages={69--84},
	year={1985},
	publisher={Elsevier}
}

@article{chen1998global,
	title={Global convergence of {O}ja's subspace algorithm for principal component extraction},
	author={Chen, Tianping and Hua, Yingbo and Yan, Wei-Yong},
	journal={IEEE Trans. Neural Netw.},
	volume={9},
	number={1},
	pages={58--67},
	year={1998},
	publisher={IEEE}
}

@article{dai2020,
	author = {Dai, Xiaoying and Wang, Qiao and Zhou, Aihui},
	title = {Gradient flow based {K}ohn-{S}ham density functional theory model},
	journal = {Multiscale Model. Simul.},
	volume = {18},
	number = {4},
	pages = {1621-1663},
	year = {2020}
}

@article{dai2021convergent,
	title={Convergent and orthogonality preserving schemes for approximating the {K}ohn-{S}ham orbitals},
	author={Dai, Xiaoying and Zhang, Liwei and Zhou, Aihui},
	journal={Numer. Math. Theor. Meth. Appl.},
	volume={16},
	number={1},
	pages={112204},
	year={2023}
}

@article{dai2017conjugate,
	title={A conjugate gradient method for electronic structure calculations},
	author={Dai, Xiaoying and Liu, Zhuang and Zhang, Liwei and Zhou, Aihui},
	journal={SIAM J. Sci. Comput.},
	volume={39},
	number={6},
	pages={A2702--A2740},
	year={2017},
	publisher={SIAM}
}

@article{Dai2021,
	author    = {X. Dai and Z. Liu and X. Zhang and A. Zhou},
	title     = {A parallel orbital-updating based optimization method for electronic structure calculations},
	journal   = {J. Comput. Phys.},
	volume    = {445},
	year      = {2021},
	pages     = {110622}
}

@article{Zhang2014,
	author    = {X. Zhang and J. Zhu and Z. Wen and A. Zhou},
	title     = {Gradient type optimization methods for electronic structure calculations},
	journal   = {SIAM J. Sci. Comput.},
	volume    = {36},
	year      = {2014},
	pages     = {265--289}
}

@article{Zhao2015,
	author    = {Z. Zhao and Z. Bai and X. Jin},
	title     = {A {R}iemannian {N}ewton algorithm for nonlinear eigenvalue problems},
	journal   = {SIAM J. Matrix Anal. Appl.},
	volume    = {36},
	year      = {2015},
	pages     = {752--774}
}

@article{Bao2004,
	author = {Bao, W. and Du, Q.},
	title = {Computing the ground state solution of {B}ose–{E}instein condensates by a normalized gradient flow},
	journal = {SIAM J. Sci. Comput.},
	volume = {25},
	year = {2004},
	pages = {1674--1697}
}

@article{edelman1998geometry,
	title={The geometry of algorithms with orthogonality constraints},
	author={Edelman, Alan and Arias, Tom{\'a}s A and Smith, Steven T},
	journal={SIAM J. Matrix Anal. Appl.},
	volume={20},
	number={2},
	pages={303--353},
	year={1998},
	publisher={SIAM}
}

@book{evans2022partial,
	author={Evans, Lawrence C.},
	title={Partial Differential Equations},
	volume={19},
	publisher={American Mathematical Society},
	year={2022},
	address={Providence}
}

@article{simon1986compact,
	author  = {Simon, Jacques},
	title   = {Compact Sets in the Space {$L^p(0,T;B)$}},
	journal = {Annali di Matematica Pura ed Applicata},
	volume  = {146},
	number  = {1},
	pages   = {65--96},
	year    = {1986}
}

@book{zeidler2013nonlinear,
	title     = {Nonlinear Functional Analysis and Its Applications: II/B: Nonlinear Monotone Operators},
	author    = {Zeidler, Eberhard},
	year      = {2013},
	publisher = {Springer Science \& Business Media},
	address   = {New York},
	series    = {Nonlinear Functional Analysis and Its Applications},
	volume    = {II/B}
}

@book{zeidler1989nonlinear,
	author    = {Zeidler, Eberhard},
	title     = {Nonlinear Functional Analysis and Its Applications II/A: Linear Monotone Operators},
	publisher = {Springer-Verlag},
	year      = {1989},
	address   = {New York},
	series    = {Nonlinear Functional Analysis and Its Applications},
	volume    = {II/A}
}

@book{lionsmagenes1972,
	author = {Lions, J.-L. and Magenes, E.},
	title = {Non-Homogeneous Boundary Value Problems and Applications},
	volume = {I},
	publisher = {Springer-Verlag},
	year = {1972},
	address = {New York}
}

@book{ambrosio2008gradient,
 	title={Gradient Flows: In Metric Spaces and in the Space of Probability Measures},
 	author={Ambrosio, Luigi and Gigli, Nicola and Savar{\'e}, Giuseppe},
 	year={2008},
 	publisher={Springer Science \& Business Media}
 }

@article{bottou2018optimization,
 	title={Optimization Methods for Large-Scale Machine Learning},
 	author={Bottou, L{\'e}on and Curtis, Frank E and Nocedal, Jorge},
 	journal={SIAM Rev.},
 	volume={60},
 	number={2},
 	pages={223--311},
 	year={2018},
 	publisher={SIAM}
 }

@article{yan1994global,
	title={Global analysis of {O}ja's flow for neural networks},
	author={Yan, Wei-Yong and Helmke, Uwe and Moore, John B},
	journal={IEEE Trans. Neural Netw.},
	volume={5},
	number={5},
	pages={674--683},
	year={1994},
	publisher={IEEE}
}

@article{schneider2009direct,
	title={Direct minimization for calculating invariant subspaces in density functional computations of the electronic structure},
	author={Schneider, Reinhold and Rohwedder, Thorsten and Neelov, Alexey and Blauert, Johannes},
	journal={J. Comput. Math.},
	volume = {27},
	pages={360--387},
	year={2009},
	publisher={JSTOR}
}

\end{document}